\documentclass[reqno, 10pt]{amsart}
\usepackage{amsfonts,amscd,amssymb,amsmath,amsthm}
\usepackage[active]{srcltx}
\usepackage[dvips]{graphicx}
\usepackage{xcolor}
\usepackage{caption,subcaption,mathrsfs,appendix}
\usepackage{xparse,tikz}
\usepackage{epstopdf}
\usepackage[T1]{fontenc}

\usetikzlibrary{matrix,calc,arrows,positioning}

\newtheorem{Proposition}{Proposition}[section]
\newtheorem{Corollary}{Corollary}[section]
\newtheorem{Lemma}{Lemma}[section]
\newtheorem{Theorem}{Theorem}[section]
 \newtheorem*{theorem*}{Theorem}

\newtheorem{Example}{Example}[section]

\theoremstyle{definition}

\theoremstyle{definition}
\newtheorem{Remark}{Remark}

\usepackage{enumerate} 

\newcommand{\R}{\mathbb{R}}
\newcommand{\C}{\mathbb{C}}

\newcommand{\Q}{\mathbb{Q}}

\newcommand{\N}{\mathbb N}

\newcommand{\Def}{\ensuremath{:=}}

\renewcommand{\Pr}{\mathbb{P}}

\newcommand{\Exp}{\mathbb{E}}

\newcommand{\filt}{\mathscr{F}}
\DeclareDocumentCommand \one { o }
{%
\IfNoValueTF {#1}
{\mathbf{1}  }
{\mathbf{1}\left\{ {#1} \right\} }%
}

\newcommand{\dtv}{d_{TV}}
\newcommand{\dbl}{d_{\text{BL}}}

\DeclareDocumentCommand \vso { o }
{%
\IfNoValueTF {#1}
{\mathcal{V}  }
{\mathcal{V}\left( {#1} \right) }%
}

\DeclareDocumentCommand \deg { O{ } }
{ \operatorname{deg}_{ #1 }}

\def\RR{\mathbb{R}}
\def\ZZ{\mathbb{Z}}
\def\NN{\mathbb{N}}

\newcommand{\alphabet}{\mathcal{A}}
\newcommand{\words}{\mathcal{A}^{*}}
\newcommand{\sequences}{\mathcal{A}^{\mathbb{N}}}
\newcommand{\statespace}{\ensuremath{\mathcal{X}}}
\newcommand{\pathspace}{\ensuremath{\mathcal{X}^*}}
\newcommand{\tm}{\ensuremath{\mathfrak{p}}}
\newcommand{\sm}{\ensuremath{\mathfrak{m}}}
\newcommand{\cha}{\ensuremath{\mathfrak{w}}}
\newcommand{\sma}{\ensuremath{\mathfrak{n}}}
\newcommand{\bx}{\ensuremath{\mathbf{x}}}

\DeclareDocumentCommand \incidence{ O{\cdot} O{} }
{
        \IfNoValueTF {#2}
        { \ell({#1}) }
        { (\ell({#1}))_{ {#2} } }
}

\DeclareDocumentCommand \functional{ O{f} }
{
        S_{ {#1} }
}

\DeclareDocumentCommand \upm{ O{y} O{p} }
{
\operatorname{UPM}_{ {#1}, {#2} }
}
\DeclareDocumentCommand \pathspace {O{p}}
{\ensuremath{\mathcal{X}^{*,#1}}}

\DeclareDocumentCommand \mpm{ O{p} }
{
\operatorname{SMPM}_{ {#1} }
}
\DeclareDocumentCommand \champm{ O{p} }
{
\operatorname{MPM}_{ {#1} }
}
\DeclareDocumentCommand \rmpm{ O{\mathfrak{a}} O{\infty} }
{
  \operatorname{RMPM}_{ {#1}, {#2} }
}

\DeclareDocumentCommand \ssim{ O{*} O{w_0} O{p} }
{
        \operatorname{SSIM}^{{#3}}_{ {#1}, {#2} }
}
  
\begin{document}

\title{Birkhoff sum fluctuations in susbstitution dynamical systems}
\author{Elliot Paquette}
\author{Younghwan Son}
\address{Department of Mathematics, Weizmann Institute of Science}
\email{paquette@weizmann.ac.il}
\email{younghwan.son@weizmann.ac.il}
\thanks{
EP is partially supported by NSF Postdoctoral Fellowship DMS-1304057.
}
\date{\today}
\maketitle

\begin{abstract}
We consider the deviation of Birkhoff sums along fixed orbits of substitution dynamical systems. We show distributional convergence for the Birkhoff sums of eigenfunctions of the substitution matrix. For non-coboundary eigenfunctions with eigenvalue of modulus $1$, we obtain a central limit theorem. For other eigenfunctions, we show convergence to distributions supported on Cantor sets. We also give a new criterion for such an eigenfunction to be a coboundary, as well as a new characterization of substitution dynamical systems with bounded discrepancy.
\end{abstract}

\section{Introduction}

Let $\alphabet$ be a finite set of letters.  Let $\words$ be the collection of all finite words using letters from $\alphabet.$ Let $\theta$ be a \emph{substitution} on $\alphabet,$ i.e.\,a map from $\alphabet \to \words.$  This can be extended to a map from $\words \to \words$ by concatenation, i.e.\,for all $a_1\cdots a_k \in \words,$ we define
\[
        \theta(a_1a_2\cdots a_k) = \theta(a_1)\theta(a_2)\cdots \theta(a_k).
\]
Define $\sequences$ to be all the sequences using elements of $\alphabet$, and we can extend $\theta$ further to map from $\sequences \to \sequences,$ again by concatenation.  
Also, for any finite or infinite word $u = u_1u_2\cdots,$ we let 
\(
        u_{< k} = u_1u_2\cdots u_{k-1},
\)
with $u_{<1}$ the empty word.  We further define $u_{\leq k}$ analogously.

Define a map $\incidence : \words \to \R^{\alphabet}$ which for any word $w = a_1a_2\cdots a_k$ and any $a \in \alphabet,$ 
\[
        \incidence[w][a] = 
        \left|\left\{ 1 \leq i \leq k\,:\,a_i = a \right\} \right|.
\]
Define the \emph{$\theta$-matrix} $M$ associated to $\theta$ as the $|\alphabet| \times |\alphabet|$ integer valued matrix so that $M_{a,b} = \incidence[\theta(b)][a]$ for all $a,b \in \alphabet,$ that is the number of occurrences of $a$ in $\theta(b).$
A substitution is called \emph{primitive} if there is a number $k > 0$ so that $M^k_{a,b} > 0$ for all $a,b \in \alphabet.$  We will assume from here on that $\theta$ is a primitive substitution.

An infinite sequence $u= (u_n)_{n=1}^{\infty} \in \sequences$ is called a \emph{fixed point} if there is $k \in \mathbb{N}$ such that $\theta^k(u) = u$. In general, one can find $a \in \alphabet$ and $k \in \mathbb{N}$ such that $\theta^k(a)$ begins with $a$. It is easy to check that $u = \lim_{m \to \infty} \theta^{km} (a)$ is a fixed point.

For any function $f: \alphabet \to \mathbb{C},$ define the map $\functional : \words \to \C$ by the rule that for any $w=a_1 a_2 \cdots a_k \in \words,$ 
\[
        \functional(w) = f(a_1) + f(a_2) + \cdots + f(a_k).
\]
Substitution systems (see Section \ref{sec:dynamics} for more background), are uniquely ergodic, and hence for a fixed point $u,$ we have by the Birkhoff ergodic theorem that 
\[
  \lim_{N\to \infty} \frac{\functional( u_{\leq N})}{N} \to \int f d \mu,
\]
for some measure $\mu.$  

We will study, in a sense, the first order correction term to this convergence. 
As a motivating example, consider the case of the irrational circle rotation, 
let $\alpha$ be any irrational number and consider any interval $I=[0,x)$ with $0 < x < 1$. Write
\[
  Z_{\alpha}(N;I) = \sum_{1 \leq n \leq N} \one[ 0 \leq \left\{ n\alpha \right\} < x ].
\]
Then by unique ergodicity of the irrational rotation we have that 
\[
  \frac{1}{N} Z_{\alpha}(N;I) \to x.
\]
The fluctuations of this ergodic average from $x$ can be described by the following theorem of Beck.
\begin{Theorem} [Beck \cite{Beck1,Beck2}]
Suppose that $\alpha$ is a quadratic irrational and $I=[0,x)$ has a rational endpoint $x$.
There are constants $C_1 = C_1(\alpha, x)$ and $C_2 = C_2(\alpha, x)$ such that for any real numbers $- \infty < t < \infty$,
\[
\frac{1}{N} \left| \{ 1 \leq n \leq N:  \frac{(Z_{\alpha} (n;I) - nx) - C_1 \log N}{C_2 \sqrt{\log N}} \leq t \} \right| \rightarrow \frac{1}{\sqrt{2 \pi}} \int_{- \infty}^t e^{-x^2/2} dx.
\]
\end{Theorem}
\noindent  This is to say that the fluctuations of the Birkhoff sum are asymptotically normally distributed. Recently, a new dynamical proof of the above theorem for the case $x = \frac{1}{2}$ is obtained by studying renormalization properties of the linear flow on an infinite staircase \cite{ADDS}.


In the case of a fixed point of a substitution, we show a central limit theorem for eigenfunctions $f$ of $M$ with eigenvalues $\lambda_f$ of modulus $1.$  For some eigenfunctions $f,$ it is possible that $f$ is a \emph{coboundary}, meaning that $\{\functional(u_{\leq n})\}_{n=1}^\infty$ is bounded.  For these $f,$ no central limit theorem is possible, and we give a characterization of eigenfunctions $f$ that have this property in Proposition~\ref{prop:coboundary}.  Otherwise, if $f$ is not a coboundary, appropriately scaling $\functional(u_{\leq n}),$ the fluctuations of the Birkhoff sum will also be asymptotically normal.
We begin by giving the easiest of our theorems to formulate, where $\lambda_f =1$ (see Section \ref{sec:results} for the full formulation). 
\begin{Theorem}
\label{thm:clt0}
Let $f$ be a left eigenfunction of $M$ with eigenvalue $\lambda_f=1$ so that $f$ is not a coboundary.  There are constants $c_1, c_2$ so that for all real $t,$
\[
  \lim_{N \to \infty} \frac{1}{N}
  \left| \{
  1 \leq n \leq N : 
  \frac{\functional(u_{\leq n}) - c_1 \log_\lambda(N)}{c_2\sqrt{\log_\lambda(N)}} \leq t
  \} \right|
  \to
  \int_{-\infty}^t \frac{e^{-x^2/2}}{\sqrt{2\pi}}\,dx.
\]
(Here $\lambda$ is the Perron-Frobenius eigenvalue of the $\theta$-matrix $M$.)
\end{Theorem}

The condition that $f$ is an eigenfunction of modulus $1$ is essential to this theorem.  Indeed, if $f$ is an eigenfunction of modulus not equal to $1,$ the asymptotic distribution of the Birkhoff sums is non-normal (see Theorems \ref{thm:<1} and \ref{thm:>1}).  Conversely, for any eigenfunction $f$ with eigenvalue of modulus $1$ which is not a coboundary, we show a central limit theorem.  

The reason eigenfunctions of $M$ play a special role here is that they satisfy a certain renormalization identity.
Specifically, for an eigenfunction $f$ of $M$ with eigenvalue $\lambda_f,$ we have that for any word $w:$
\begin{equation}
  \label{eq:renormalization}
  \functional(\theta^k(w)) = \lambda_f^k\functional(w)
\end{equation}
Hence in the case where $\lambda_f$ has modulus $1,$ words of all different scales have the same contribution to the Birkhoff sum, due to which we can eventually prove a central limit theorem.

Let $(X_{\theta},T)$ be the substitution subshift of bi-infinite sequences associated to the primitive substitution $\theta$ (see Section~\ref{sec:dynamics} for background). It is known that if $\theta$ is primitive, there exists a unique ergodic measure $\mu$ on $X_{\theta}$. 
We are also interested in studying the behavior of ergodic sums for any point $x \in X_{\theta}$. Indeed we show a central limit theorem for a left eigenfunction $f$ of $M$, where $f$ is not a coboundary and the corresponding eigenvalue $\lambda_f$ is of modulus $1$ (see Theorem \ref{thm:clt2}). This theorem is intriguing in that the behavior of the Birkhoff sums of typical points is different from fixed points.  

Another approach taken by Bressaud, Bufetov and Hubert to studying fluctuations of Birkhoff sums is to look at $\functional( v_{\leq N_\ell})$ for some sequence $N_\ell\to \infty,$ where $v$ is distributed randomly according to $\mu.$  In this case, it turns out that this distribution may depend on the sequence of $N_\ell$ chosen: see \cite{BBH}.


\section{Preliminaries}
\subsection{Probability background}

We will occasionally use probability formalism where convenient.  We may say, for example, that $Z$ is a real valued random variable with a standard normal distribution, without specifying the probability space or naming the probability measure.  In this case, we are only interested in distributional properties of $Z.$  We will use $\Pr$ as a placeholder for this measure, for example:
\[
  \Pr [ Z \in A ] = \int\limits_A \frac{e^{-x^2/2}}{\sqrt{2\pi}}\,dx.
\]
We will also use $\Exp$ to denote integration against the distribution of a random variable, for example
\[
  \Exp f(Z) = \int\limits_\R f(x)\frac{e^{-x^2/2}}{\sqrt{2\pi}}\,dx.
\]

Our main theorems are stated in terms of distributional convergence of random variables, or \emph{weak} convergence.  There are many equivalent definitions of weak convergence, which we will freely interchange as convenient.  The equivalence of these definitions usually goes by the name of the \emph{Portmanteau} lemma.

\begin{Lemma}[Portmanteau lemma]
  Let $E$ be a metric space and let $\mu,\mu_1,\mu_2,\dots$ be Borel probability measures on $E.$  The following are equivalent.  If they occur, we write $\mu_n \Rightarrow \mu$ and say that $\mu_n$ converges in distribution to $\mu.$
  \begin{enumerate}
    \item For all bounded continuous functions $f,$
      \(
	\int f d\mu_n \to \int f d\mu.
      \)
    \item For all Lipschitz continuous functions $f,$
      \(
	\int f d\mu_n \to \int f d\mu.
      \)
    \item For all measurable $A$ with $\mu(\partial A) = 0$, 
      \(
      \lim_{n \to \infty} \mu_n(A) = \mu(A).
      \)
  \item If $E = \R^d,$ then for all $A$ of the form 
  \((-\infty, x_1] \times \cdots \times (-\infty, x_d] \)
  with $\mu(\partial A) = 0,$
  \(
      \lim_{n \to \infty} \mu_n(A) = \mu(A).
  \)
  \end{enumerate}
  \label{lem:portemanteau}
\end{Lemma}
\noindent See \cite[Theorem 13.16]{Klenke}.

It is also possible to metrize weak convergence.  We will make use of such a metric later on.  Define the \emph{bounded-Lipschitz metric} on the space of probability measures on a metric space $E$ as
\[
	\dbl(\mu,\nu)
	= \sup_{\substack{ f: E \to \R \\ \|f\|_\infty \leq 1 \\ \|f\|_{\text{Lip}} \leq 1}} 
	\left|
	\int f\,d\mu
	-\int f\,d\nu
	\right|.
\]
When $E$ is a separable metric space, it is a theorem of Dudley~\cite{Dudley} that weak convergence is equivalent convergence in the bounded-Lipschitz metric.

Theorems \ref{thm:clt1}, \ref{thm:clt2}, and \ref{thm:Markov} show distributional convergence to a complex normal distribution.  A complex random variable $Z$ with mean $0$ is said to have a complex normal distribution with positive definite covariance matrix 
\(
\Gamma = \left[\begin{smallmatrix}
  \Exp (\Re Z)^2 & \Exp(\Re Z \Im Z) \\
  \Exp (\Re Z\Im Z) & \Exp (\Im Z)^2 \\
\end{smallmatrix}\right]
\)
if for any Borel measurable $A\subseteq \R^2$
\[
  \Pr\left[ Z \in A \right]
  = \int\limits_A \frac{1}{2\pi\sqrt{\det \Gamma}}e^{-\tfrac 12 x^t \Gamma^{-1} x}\,dx_1dx_2.
\]
It follows from this definition that both of $\Re Z$ and $\Im Z$ are normally distributed.  Note that $Z$ is completely determined by $\Exp |Z|^2$ and $\Exp Z^2.$  If $\Exp \Re Z\Im Z = 0,$ then the real and imaginary parts of $Z$ are in fact independent normal random variables.

\subsection{Path space}
Define the \emph{state space} $\statespace$ by
\begin{equation}
        \statespace
        =
        \left\{ (a,j) : a \in \alphabet, 1 \leq j \leq |\theta(a)| \right\}.
        \label{eq:ss}
\end{equation}
Also, for each $p \in \NN$, define the \emph{path space} $\pathspace[p] \subseteq \statespace^{p}$ of all sequences $\left\{ (v_i,k_i) \right\}_{i=1}^p$ so that for all $1\leq i < p,$ $v_i = \theta(v_{i+1})_{k_{i+1}}$.

We will now define a coding of the collection of strict prefixes of $w=\theta^{p}(a).$
Let $n \leq |\theta^{p}(a)|$ be a positive integer. 
One can find $v_p, v_{p-1}, \dots, v_1 \in \alphabet$ and positive integers $k_p, k_{p-1}, \dots, k_1$ such that
\[
          w_{[1,n)} = \theta^{p-1} (\theta (v_p)_{< k_p}) \cdot \theta^{p-2} (\theta (v_{p-1})_{<k_{p-1}}) \cdots \theta(\theta(v_2)_{< k_2}) \cdot \theta(v_1)_{< k_1},      
\] 
where 
\[
      v_p = a, \quad 1 \leq k_p \leq |\theta(a)|
\]
and for $ i = 1, \dots, p-1$
\[
      v_i = \theta(v_{i+1})_{k_{i+1}}, \quad 1 \leq k_i \leq |\theta(v_i)|.
\]
Since this expression is unique, there is an injection $\Psi_{a,p} : \left\{ 1,2,\ldots, |\theta^p(a)| \right\} \to \pathspace[p]$ given by
\begin{equation}
        \label{eq:11}
        \Psi_{a,p}(n) = (v_1, k_1) (v_{2}, k_{2}) \cdots (v_p, k_p).
\end{equation}

\begin{figure}
\begin{tikzpicture}
\matrix (network)
[matrix of nodes,%
nodes in empty cells,
nodes={outer sep=0pt,circle,minimum size=4pt,draw},
column sep={1cm,between origins},
row sep={2cm,between origins}]
{
   & & & & & \\
   & & & & & \\
   & & & & & \\
   & & & & & \\
};
\foreach \a in {1,...,6}{
  \draw [dotted] (network-1-\a)  -- (network-2-\a);
	\draw [stealth-] ([yshift=-5mm]network-4-\a.south) -- (network-4-\a);
}

\foreach \a in {1,2,6}{
	\foreach \b in {1,2,3} {
		\draw (network-4-\a) -- (network-3-\b);
		\draw (network-3-\a) -- (network-2-\b);
}
}
\foreach \a in {3,4,5}{
	\foreach \b in {4,5,6} {
		\draw (network-4-\a) -- (network-3-\b);
		\draw (network-3-\a) -- (network-2-\b);
}
}

\node[left=of network-1-1] { $1$};
\node[left=of network-2-1] { $i-1$};
\node[left=of network-3-1] { $i$};
\node[left=of network-4-1] { $i+1$};

\node[above=of network-1-1] {$(a,1)$};
\node[above=of network-1-2] {$(a,2)$};
\node[above=of network-1-3] {$(a,3)$};
\node[above=of network-1-4] {$(b,1)$};
\node[above=of network-1-5] {$(b,2)$};
\node[above=of network-1-6] {$(b,3)$};

\end{tikzpicture}
\caption{ Here we give a graphical representation of path space for the substitution $\theta(a)=aab,\theta(b)=bba.$}
\end{figure}
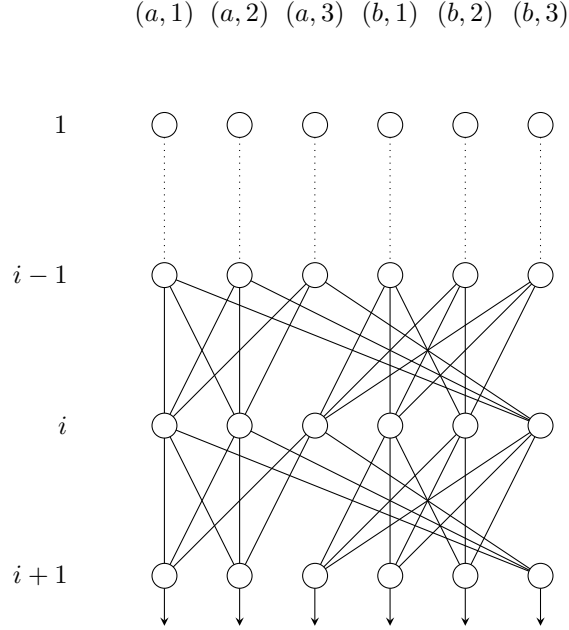

We also define the infinite path space $\pathspace[\infty] \subseteq \statespace^{\infty}$ as the collection of all sequences $\left\{ (v_i,k_i) \right\}_{i=1}^\infty$ so that for all $1\leq i < p,$ $v_i = \theta(v_{i+1})_{k_{i+1}}.$  On this space, there is a natural Markov measure associated to $\theta,$ which we denote $\mpm[\infty].$  It is a primitive, stationary Markov chain with some transition matrix $\tm.$  It can be defined combinatorially (see \eqref{eq:p}) or as the unique invariant measure under the \emph{adic transformation} (see Section \ref{sec:dynamics}).

For every $x=(a,j),y=(b,k) \in \statespace,$ 
let $\ssim[x][y][1]=\one[a = \theta(b)_{k}].$  We then define, inductively,
\[
\ssim[x][y][p] =\sum_{z \in \statespace} \ssim[x][z][p-1]\ssim[z][y][1].
\]
Hence, $\ssim[x][y]$ counts the number of elements in $\pathspace[p+1]$ started from $x$ and ended at $y.$  Let $\ssim[*][y] = \sum_{x \in \statespace} \ssim[x][y].$

We can see that $\ssim[\cdot][\cdot][1]$ is a primitive matrix:
since $\theta$ is primitive, there exists $k$ such that for any $c, d \in \alphabet$, $c$ appears in $\theta^k(d)$. This means that there exist $m_1, m_2, \dots, m_k$ so that $\theta(\theta( \cdots (\theta(d)_{m_1})_{m_2} \cdots )_{m_k} = c$. Thus for any $(a, i), (b, j) \in \statespace$, $a = \theta( \cdots (\theta(b)_j)_{m_1} \cdots)_{m_k}$, so $\ssim[\cdot][\cdot][k+1] \geq 1$.

\subsection{Eigenfunctions of $\ssim[\cdot][\cdot][1]$}
We begin by observing that for all $(a,j) \in \statespace$ and all $b \in \alphabet$
\[
        M_{a,b}
        =\sum_{k=1}^{|\theta(b)|}
        \ssim[(a,j)][(b,k)][1].
\]
Let $(\rho(b))_{b}$ be a right Perron-Frobenius eigenfunction of $M$ with Perron-Frobenius eigenvalue $\lambda,$ and define the map ${\hat \rho}( (b,k)) = \rho(b)$ for all $(b,k) \in \statespace.$  Then we have that for any $(a,j) \in \statespace,$
\begin{align*}
        \lambda {\hat \rho}(a,j)
        &= \lambda \rho(a) \\
        &= \sum_{b}M_{a,b}\rho(b) \\
        &=\sum_{(b,k) \in \statespace}\ssim[(a,j)][(b,k)][1]\rho(b) \\
        &=\sum_{(b,k) \in \statespace}\ssim[(a,j)][(b,k)][1]{\hat \rho}(b,k).
\end{align*}
Hence, as it is non-negative, $\hat \rho$ is a right Perron-Frobenius eigenfunction of $\ssim[\cdot][\cdot][1]$ with eigenvalue $\lambda;$ the same calculation shows that any right eigenfunction of $M$ can be canonically associated to a right eigenfunction of $\ssim[\cdot][\cdot][1].$  

Conversely, a similar calculation shows that any left eigenfunction $\hat g$ of $\ssim[\cdot][\cdot][1]$ gives rise to a left eigenfunction of $M$ with the same eigenvalue by the formula $g(a) = \sum_{j=1}^{|\theta(a)|} \hat g (a,j),$ provided that $g$ is non-zero.  
Furthermore, by the eigenvector equation we have
\begin{align*}
  \lambda_g \hat g(b,k)
  &=\sum_{(a,j) \in \statespace} \hat g(a,j) \ssim[(a,j)][(b,k)][1] \\
  &=\sum_{(a,j) \in \statespace} \hat g(a,j) \one[\theta(b)_k = a] \\
  &=\sum_{a \in \alphabet} g(a) \one[\theta(b)_k = a] \\
  &=g(\theta(b)_k).
\end{align*}

Hence, if $\hat \sigma$ is the left Perron-Frobenius eigenfunction of $\ssim[\cdot][\cdot][1],$ normalized so that 
\(
\sum_{x \in \statespace} 
{\hat \sigma}(x)
{\hat \rho}(x)
=1,\)
then $\sigma(a) = \sum_{j=1}^{|\theta(a)|} \hat \sigma (a,j)$ is the left Perron-Frobenius eigenfunction of $M$ normalized so that 
\(
\sum_{a \in \alphabet} 
{\sigma}(a)
{\rho}(a)
=1\)
and $\hat \sigma(b,k) = \lambda^{-1} \sigma(\theta(b)_k).$

\subsection{Measures on path space}
\label{sec:measures}
To prove our theorems, we must consider the measure $\nu_N$ that arises on $\pathspace$ by choosing an integer $n \in \left\{ 1,2,\ldots,N \right\}$ uniformly at random. This means that for every $n \in \{1, 2, \dots, N\}$
\[\nu_N (\{\psi_{a,p}(n)\}) = \frac{1}{N}, \]
where $|\theta^{p-1}(a)| < N \leq |\theta^p(a)|$.

We will decompose $\nu_N$ in terms of other measures. 
For every $p \in \NN$ and $y \in \statespace$
define the probability measure $\upm$
for each
$\bx=x_1x_2\cdots x_p \in \pathspace$ by
\[
        \upm\left( \left\{ \bx \right\}\right) =
        \frac{\one[x_{p} = y]}{\ssim[*][y][p-1]}.
\]
Then the following is immediate.
\begin{Proposition}
	Let $a \in \alphabet$ and $p \in \NN$ be such that $|\theta^{p-1}(a)| < N \leq |\theta^p(a)|.$
        Let $\Psi_{a,p}(N) = (v_1,k_1)(v_2,k_2)\cdots (v_p,k_p).$
        Then for every $\bold{x}=x_1x_2\cdots x_p \in \pathspace,$ we have
        \[
                \nu_N
                \left( \left\{ \bold{x} \right\} \right)
                =
                \sum_{\ell=1}^p \sum_{1 \leq j < k_\ell} 
                \frac{\ssim[*][(v_\ell,j)][\ell]}{N} 
                \upm[(v_\ell,j)][\ell]
                \left( \left\{ x_1x_2\cdots x_\ell \right\} \right)
                .
        \]
        \label{prop:decomposition}
\end{Proposition}
We can also write, for any $\bx = x_1x_2\dots x_p \in \statespace^p,$ that
\[
\upm\left( \left\{ \bx \right\}\right) 
= 
        \frac{\one[x_{p} = y]}{\ssim[*][y][p-1]}
	\prod_{i=1}^{p-1}
	\ssim[x_i][x_{i+1}][1].
\]
This family of measures has a (inhomogeneous) Markov chain structure.  To express this, define a collection $h^y_p$ of Markov transition matrices given by, for all $x,z \in \statespace$, 
\[
        h^y_p(x,z) \Def
        \frac{\ssim[x][z][1]\ssim[z][y][p-2]}{\ssim[x][y][p-1]} ,
\]
which is a transition matrix as for any $x,y \in \statespace$ and any $p \in \NN,$
\[
\sum_{z \in \statespace} h^y_p(x,z)
=
\sum_{z \in \statespace} \frac{\ssim[x][z][1]\ssim[z][y][p-2]}{\ssim[x][y][p-1]}=1.
\]
This allows us to express $\upm$ for any $\bold{x}=x_1 x_2 \cdots x_p \in \statespace^{p}$ as
\begin{align}
  &\upm (\{\bold{x}\}) = \nonumber\\
        &\frac{\ssim[x_1][y][p-1]}{\ssim[*][y][p-1]}
	h^y_p(x_1,x_2)
	h^y_{p-1}(x_2,x_3)
	\dots
	h^y_{3}(x_{p-2},x_{p-1})
	\frac{\ssim[x_{p-1}][x_p][1]}{\ssim[x_{p-1}][y][p-1]}
	\one[x_p = y].
  \label{eq:upm}
\end{align}

By the Perron-Frobenius theorem, as $\ssim[\cdot][\cdot][1]$ is primitive, we have for all $x,y \in \statespace$
\begin{equation}
  \label{eq:ssimlimit}
        \lim_{p \to \infty} \lambda^{-p}\ssim[x][y][p] = 
        {\hat \sigma}(y)
        {\hat \rho}(x).
      \end{equation}
Hence the limit
\[
        \tm(x,z) = \lim_{p\to \infty} h^y_p(x,z)
\]
exists for all $x,z \in \statespace$ and is independent of $y.$ In fact, there is a constant $c > 0$ so that for all $p \in \NN,$
\begin{equation}
        \sup_{x,z\in\statespace} | \tm(x,z) - h^y_p(x,z)| \leq e^{-cp}.
        \label{eq:pconvergence}
\end{equation}
Further, we have the following explicit formula for $\tm:$
\begin{equation}
        \tm(y,z) = \frac{\ssim[y][z][1]{\hat \rho}(z) }{\lambda {\hat \rho}(y)}.
        \label{eq:p}
\end{equation}

It also follows that the limit
\[
	\cha(x_1) = \lim_{p  \to \infty}
\frac{\ssim[x_1][y][p-1]}{\ssim[*][y][p-1]}
= \frac{{\hat \rho}(x_1)}{\sum_{x \in \statespace} {\hat \rho}(x)}
\]
exists.
This motivates the definition of the following Markov measure on $\pathspace,$ where for any $\bx = x_1x_2\dots x_p,$
\begin{equation}
	\champm\left( \left\{ \bx \right\} \right)
		=
	        \cha(\{x_1\})
	\prod_{i=1}^{p-1} \tm(x_{i},x_{i+1}).
	\label{eq:champm}
\end{equation}
We will also define a stationary version of $\champm.$

Let $\sm$ be the invariant measure of $\tm$ on $\statespace.$
Define, for $\bold{x}=x_1 x_2 \cdots x_p  \in \pathspace,$
\begin{equation}
        \mpm\left( \left\{ \bold{x} \right\} \right)
        =
        \sm(\{x_1\})
	\prod_{i=1}^{p-1} \tm(x_{i},x_{i+1})
	.
	\label{eq:mpm}
\end{equation}
It follows that the invariant measure $\sm$ must be $\sm({y})={\hat \sigma}(y){\hat \rho}(y),$ as
\[
        \sum_{y \in \statespace}
{\hat \sigma}(y){\hat \rho}(y)
\tm(y,z)
=
        \sum_{y \in \statespace}
\frac{{\hat \sigma}(y) \ssim[y][z][1]
\hat \rho(z)
}{\lambda }
=
{\hat \sigma}(z){\hat \rho}(z).
\]


Given a function $f : \alphabet \to \C,$
define a probability measure $\sma$ on $\alphabet$ by the following formula
\begin{equation}
        \int_{\alphabet} f(a) \sma(da)
        =\int_{\statespace} \functional[f]( \theta(a)_{<j}) \sm(d(a,j))
        =
        \sum_{(a,j) \in \statespace}
{\hat \sigma}(a,j){\hat \rho}(a,j)
\functional[f]( \theta(a)_{<j}).
\label{eq:com}
\end{equation}
This measure will be used to express the drift in the central limit theorems as well as the condition for being a coboundary.

\begin{Example}
\label{ex:main} 
Let $\alphabet = \{a, b\}$ and $\theta$ is given by
\[\theta: a \to aab, \quad b \to bba. \]
Then the $\theta$-matrix $M$ is 
\[ \left( \begin{array}{ccc}
2 & 1  \\
1 & 2  \end{array} \right).\]

\begin{itemize}
\item Eigenvalues: $ \lambda =3, \lambda_f = 1.$
\item Left eigenvectors  $\sigma = [1,1], \quad f = [1,-1].$
\item Right Perron-Frobenius eigenvector $\rho = \left(\begin{smallmatrix} 1 \\ 1\end{smallmatrix}\right).$ 
\item The state space is denoted by $\statespace = \{(a,1), (a,2), (a,3), (b,1), (b,2), (b,3)\}.$ Then 
  $\ssim[\cdot][\cdot][1]$ 
  is given by
\[ \ssim[\cdot][\cdot][1] =
 \left( \begin{array}{cccccc}
1 & 1 & 1 & 0 & 0 & 0  \\
1 & 1 & 1 & 0 & 0 & 0  \\
0 & 0 & 0 & 1 & 1 & 1  \\
0 & 0 & 0 & 1 & 1 & 1  \\
0 & 0 & 0 & 1 & 1 & 1  \\
1 & 1 & 1 & 0 & 0 & 0  \end{array} \right)\]
\item A left and a right Perron-Frobenius eigenvectors $\hat{\sigma}$ and $\hat{\rho}$ of $\ssim[\cdot][\cdot][1]$ are  
\[ \hat{\sigma} = [1/6,1/6,1/6,1/6,1/6,1/6], \quad \hat{\rho}^t =[1,1,1,1,1,1]. \] 
\item the invariant measure is $\sm({y})={\hat \sigma}(y){\hat \rho}(y) = \frac{1}{6}$ for any $y \in \statespace$.
\item From \eqref{eq:p},
$\tm  = \frac{1}{3} \ssim[\cdot][\cdot][1].$
\end{itemize}
We have that
\begin{equation}
\label{eq:maindrift}
 \int_{\alphabet} f(a) \sma(da) = \sum_{(a,j) \in \statespace} {\hat \sigma}(a,j){\hat \rho}(a,j) \functional[f]( \theta(a)_{<j}) = (0 + 1 + 2 + 0 -1 -2) \frac{1}{6} =0.
\end{equation}
\end{Example}

\subsection{Reversed path space}

When working with eigenfunctions with eigenvalue $|\lambda_f| > 1,$ it is more convenient to work with a \emph{reversed path space}.
Define the reversed path space $\pathspace_r \subset \statespace^p$ as the reversals of all sequences in $\pathspace.$  Fix $a \in \alphabet$ and consider $\pathspace_r$ as embedded in $\statespace^{\infty}$ by appending to any element the infinite sequence $(a,1)(a,1)\cdots .$  Make $\statespace^{\infty}$ into a topological space by endowing it with the product topology, and let $\pathspace[\infty]_r$ be the closure of $\cup_{p}^\infty \pathspace_r.$  Now define $\Psi_a^r : \NN \to \statespace^{\infty}$ by
\[
  \Psi_a^r(N)
  =(v_p,k_p)(v_{p-1},k_{p-1})\cdots (v_1,k_1)(a,1)(a,1)\cdots,
\]
where $(v_i,k_i)$ are those appearing in $\Psi_{a,p}(N)$ for $|\theta^{p-1}(a)|<N \leq |\theta^p(a)|.$  

We will now give another description of $\upm$, which is useful for reversed path space.
Define a collection $q_p(y,z)$ of Markov transition matrices given by 
\[
        q_p(y,z) \Def
        \frac{\ssim[*][z][p-2]}{\ssim[*][y][p-1]} \ssim[z][y][1],
\]
which is a transition matrix as for any $y \in \statespace$ and any $p \in \NN,$
\[
\sum_{z \in \statespace} q_p(y,z)
=
\sum_{z \in \statespace} \frac{\ssim[*][z][p-2]\ssim[z][y][1]}{\ssim[*][y][p-1]}
=1.
\]
This allows us to express $\upm$ for any $\bold{x}=x_1 x_2 \cdots x_p \in \pathspace$ as
\begin{equation}
  \upm (\{\bold{x}\}) = \prod_{i=1}^{p-1} q_{p-i+1} (x_{p-i+1}, x_{p-i}) \, \one[y=x_p]. 
  \label{eq:upmr}
\end{equation}

As $\ssim[\cdot][\cdot][1]$ is a primitive matrix, by the Perron-Frobenius theorem, the limit
\begin{equation}
        \label{eq:pstar}
        \tm^*(y,z) = \lim_{p\to \infty} q_p(y,z)= \frac{{\hat \sigma}(z) \ssim[z][y][1]}{\lambda {\hat \sigma}(y)}
\end{equation}
exists for all $y,z \in \statespace.$ We also have that there is a constant $c > 0$ so that for all $p \in \NN,$
\begin{equation}
        \sup_{y,z\in\statespace} | \tm^*(y,z) - q_p(y,z)| \leq e^{-cp}.
        \label{eq:pstarconvergence}
\end{equation}

Let $\mathfrak{a}$ be any probability measure on $\statespace,$ and define a Markov measure $\rmpm$ on $\pathspace[\infty]_r$ by, for any cylinder set $[\bold{x}] = [x_1x_2\dots x_p],$
\begin{equation}
  \rmpm\left( [\bold{x}] \right)
  = \mathfrak{a}({x_1}) \prod_{i=1}^{p-1} \tm^*(x_i, x_{i+1}).
  \label{eq:rmpm}
\end{equation}

\section{Main results}
\label{sec:results}


Let $\theta$ be a primitive substitution and let $u=(u_n)_{n=1}^{\infty}$ be any fixed point of $\theta$. Denote the Perron-Frobenius eigenvalue of the $\theta$-matrix by $\lambda.$ 
For eigenfunctions $f$ of $M$ with eigenvalue $\lambda_f$ having $|\lambda_f| < 1,$ it is well known that the Birkhoff sums $\functional(u_{\leq N})$ stay bounded.  We show they also have distributional convergence to a bounded random variable.

\begin{Theorem}
\label{thm:<1}
Let $f$ be a left eigenfunction of $M$ with eigenvalue $\lambda_f$ with $|\lambda_f| < 1$. 
Let $K_N$ be a random variable with uniform distribution on $\left\{ 1,2,\ldots, N \right\}.$ For $\bold{x} = (v_1,k_1) (v_2,k_2) \cdots \in \pathspace[\infty]$, define $W_f (\bold{x}) = \sum_{i=1}^{\infty} \lambda_f^{i-1} \functional(\theta(v_i)_{< k_i}).$
Then, 
\[
  \functional(u_{\leq K_N})  \Rightarrow W_f(\bold{X}),
\]
where $\bold{X}$ has the distribution of $\champm[\infty].$
\end{Theorem}

In the case that $|\lambda_f| > 1,$ on the other hand, the Birkhoff sums will not in general have a distributional limit.  In fact, there are many distributional limit points of ${\functional(u_{\leq K_{N}})} { N^{-\log_\lambda(\lambda_f)}}$ as $N\to \infty.$  We show that by choosing different subsequences, it is possible to get different distributional limits, although their distributions are closely related.
\begin{Theorem}
\label{thm:>1}
Let $a = u_1,$
and let $f$ be a left eigenfunction of $M$ with eigenvalue $\lambda_f$ having $\lambda > |\lambda_f| > 1$.
Let $N_\ell$ be a sequence with $N_\ell \to \infty$ so that 
\[
  \Psi_a^r(N_\ell) \to \mathbf{z} = (\rho_1,\kappa_1)(\rho_2,\kappa_2)\dots \in \pathspace[\infty]_r
\]
Suppose $\hat \rho$ is normalized so that $\sum_{x \in \statespace} \hat \rho(x) =1.$
Define a probability measure on $\statespace$ by
	 \[
\mathfrak{a}( (v,k) ) = 
\frac{1}{R} {\hat \sigma}((v,k)) \sum_{q=1}^\infty \one[ v = \rho_q \text{ and } k < \kappa_q] \lambda^{1-q},
	 \]
	 where $R > 0$ is chosen so that $\mathfrak{a}$ is a probability measure.

Let $K_{N_\ell}$ be a random variable with uniform distribution on $\left\{ 1,2,\ldots, N_\ell \right\}.$ For $\bold{x} = (v_1,k_1) (v_2,k_2) \cdots \in \pathspace[\infty]$, define $U_f (\bold{x}) = \sum_{i = 1}^{\infty} \lambda_f^{-i} \functional(\theta(v_i)_{< k_i}).$  
Let $p= p(\ell) \in \NN$ be such that $|\theta^{p-1}(a)| < N_\ell \leq |\theta^p(a)|.$
Then, 
\[
  \frac{\functional(u_{\leq K_{N_\ell}})}{ N_\ell^{\log_\lambda|\lambda_f|}
e^{ip(\ell)\arg\lambda_f}
} \Rightarrow 
  \frac{U_f(\bold{X})}{R^{\log_\lambda|\lambda_f|}},
\]
where $\bold{X}$ has the distribution of
\(
\rmpm.
\)
\end{Theorem}

\begin{Remark}
	Alternatively Theorem \ref{thm:>1} can be formulated as 
	\[
		\frac{\functional(u_{\leq K_{N_\ell}})}{ \lambda_f^{p(\ell)}}
 \Rightarrow 
  \frac{U_f(\bold{X})}{R^{\log_\lambda|\lambda_f|}}.
\]

\end{Remark}

\begin{Remark}
  In both of Theorems~\ref{thm:<1} and \ref{thm:>1}, the support of the limiting measure is a Cantor set.  This can be seen by noting that $W_f$ (and $U_f$) are continuous functions from $\statespace^{\infty}$ with the product topology, which is a Cantor set.
\end{Remark}

We will soon formulate our main theorems when $|\lambda_f| = 1,$ but before doing so, we give a characterization of eigenfunction coboundaries.  Recall that a continuous function $f: X \to \C$ is called a coboundary if there exists a continuous function $g$ such that $f = g - g \circ T$. By the Gottschalk-Hedlund theorem \cite{GH}, if $X$ is a compact metric space and $T:X \rightarrow X$ is a minimal homeomorphism, and if $f:X\to \C$ is continuous, then $f$ is a coboundary if and only if there exists $K< \infty$ such that $|\sum_{n=1}^N f(T^nx)| \leq K$ for all $N \in \N$ and $x \in X$.  Kornfeld and Lin \cite{KL} obtained a more general result: if $X$ is a compact Hausdorff space, and $T$ is an irreducible Markov operator on $C(X)$, then $\sup_N \|\sum_{n=1}^N f \circ T^n \| < \infty$ if and only if $f$ is a coboundary. 

For left eigenfunctions $f$ of $M,$ we show a further characterization of coboundaries.
\begin{Proposition}
\label{prop:coboundary} Let $(X_{\theta}, \mathcal{B}, \mu, T)$ be a substitution dynamical system associated to a primitive substitution $\theta$.
        Suppose that $f$ is a left eigenfunction of $M$ with eigenvalue $\lambda_f$ having $|\lambda_f|=1.$  Then the following are equivalent.
        \begin{enumerate}[(i)]
                \item The function $w \mapsto f(w_1)$ from $X_\theta \to \C$ is a coboundary.
                \item There exists $w \in X_{\theta}$ such that $\sup_{N} |\functional[f](w_{\leq N})| <\infty.$
                \item There is a function $h : \alphabet \to \C$ so that for all $(a,j) \in \statespace$
                        \[
                                \functional[f](\theta(a)_{<j})
                                =
                                \int_{\alphabet} f(c)\sma(dc)
                                +h(a) - \lambda_f^{-1} h(\theta(a)_j).
                        \]
                \item There is a function $h : \alphabet \to \C$ so that the following hold. 
        \begin{enumerate}
                \item
                        For all $(a,j) \in \statespace$ with $1 \leq j < |\theta(a)|,$
                        \[
                                f(\theta(a)_j) = \lambda_f^{-1}(h(\theta(a)_j) - h(\theta(a)_{j+1})).
                        \]
                \item 
                        For all $a \in \alphabet,$
                        \[
                                \int_{\alphabet} f(c)\,\sma(dc)
                                + h(a) - \lambda_f^{-1} h(\theta(a)_1) = 0.
                        \]
        \end{enumerate}
        \end{enumerate}
\end{Proposition}

\begin{Remark}
        Say that a function $f : \alphabet \to S^1$ is a coboundary in the sense of Host (see \cite{Host} or \cite[Definition 7.3.13]{Fogg}) if there is a function $h : \alphabet \to S^1$ so that for all admissible $2$-letter words $ab,$ $h(b) = h(a)f(a).$
        If $f$ satisfies condition $(iv)$ of Proposition~\ref{prop:coboundary}, then $a \mapsto e^{i \Re f(a)}$ and $a \mapsto e^{i \Im f(a)}$ are coboundaries in the sense of Host.
        \label{rem:host}
\end{Remark}

For an eigenfunction $f$ of $M$ which is not a coboundary with eigenvalue $\lambda_f$ having $|\lambda_f| = 1,$ let $Z_f$ be the following normal random variable.
\begin{enumerate}
	\item If $\lambda_f \not\in \R,$ then $Z_f$ is a complex normal variable.  Further, $Z_f = X + iY,$ where $X,Y$ are independent, centered normal distributions with $\Exp X^2 = \Exp Y^2,$ and letting $g = f - \int f\,d\sma,$ we have
 \[
	\Exp |Z_f|^2 = \Exp |g(X_1)|^2 + \sum_{k=2}^\infty 2\Exp\Re[\lambda_f^{k-1}g(X_1)\overline{g(X_k)}], 
	\]
	where $(X_1,X_2,\ldots)$ has the distribution of $\mpm[\infty]$.  By Proposition \ref{prop:coboundary} and Theorem~\ref{thm:Markov}, this variance is $0$ if and only if $f$ is a coboundary. 
\item If $\lambda_f \in \R$ and $f$ is real, then we have that $Z_f$ is real and has the same variance as above.  Again, the variance is $0$ if and only if $f$ is a coboundary.
	\end{enumerate}

\begin{Theorem}
\label{thm:clt1}
Let $f$ be a left eigenfunction of $M$ with eigenvalue $\lambda_f$ with $|\lambda_f| = 1$ so that $f$ is not a coboundary. 
Let $K_N$ be a random variable with uniform distribution on $\left\{ 1,2,\ldots, N \right\}.$  Then if $\lambda_f \neq 1,$ as $N\to\infty$
\[
\frac{\functional(u_{\leq K_N})
}{\sqrt{\log_\lambda(N)}}
\Rightarrow
Z_f.
\]
If $\lambda_f = 1$ then 
\[
\frac{\functional(u_{\leq K_N})
-
\log_\lambda(N)
\int_{\alphabet} f(a)\sma(da)
 }{\sqrt{\log_\lambda(N)}}
\Rightarrow
Z_f.
\]
\end{Theorem}

The combination of Theorems \ref{thm:<1}, \ref{thm:>1}, and \ref{thm:clt1} 
allows us to give a new complete description of systems with \emph{bounded discrepancy} (the first such description is due to \cite{A}).  Say that a fixed point $u$ has bounded discrepancy if for every $a \in \alphabet,$ with $f_a : \alphabet \to \R$ given by $f_a(b) = \one[a=b],$
\[
  \sup_{N \in \NN} |\functional[{f_a}](u_{\leq N}) - Nq\left( \left\{ a \right\} \right)| < \infty,
\]
where $q$ is the occurrence frequency of $a,$\,i.e.\,\(
q(a) = \lim_{N\to\infty} N^{-1}\functional[{f_a}](u_{\leq N}).
\)
\begin{Corollary}
  Suppose that $u$ is a fixed point of a primitive substitution $\theta.$ 
  Then $u$ has bounded discrepancy if and only if 
  \begin{enumerate}
    \item All eigenvalues of $M$ except the Perron-Frobenius eigenvalue have modulus less than or equal to $1.$
    \item The geometric multiplicity of each eigenvalue of modulus $1$ equals its algebraic multiplicity, i.e.\,each Jordan block in the Jordan form of $M$ having eigenvalue of modulus $1$ is $1$-dimensional.
    \item Each eigenfunction $f$ with eigenvalue equal to $1$ is a coboundary in the sense of Proposition~\ref{prop:coboundary}.
  \end{enumerate}
\end{Corollary}
\begin{proof}
  Let $\mathcal{W} \subset \left\{ f: \int f\,dq = 0 \right\}$ be all those functions so that
  \[
    \sup_{N \in \NN} |\functional[f](u_{\leq N})| < \infty.
  \]
  Note that this is a vector space.  Hence $u$ has bounded discrepancy if and only if $\dim \mathcal{W} = |\alphabet|-1,$ as the functions $\left\{ f_a -\int f_a \,dq : a \in \alphabet \right\}$ span the space $\left\{ f: \int f\,dq = 0 \right\}.$ Given a basis of generalized eigenfunctions 
$f_0,f_1,f_2,\dots, f_r,$ with $f_0$ the Perron-Frobenius eigenfunction, we have that $\int f_i\,dq = 0,$ $(1 \leq i \leq r)$ and hence $\left\{ f: \int f\,dq = 0 \right\}$ is also spanned by $f_1,f_2,\dots,f_r.$ Hence the necessity of the first and third conditions follow by Theorems \ref{thm:>1} and \ref{thm:clt1}.  For the second condition, suppose that $\lambda_f$ were an eigenvalue with $|\lambda_f| = 1$ so that in the Jordan form of $M$, there is a nontrivial Jordan block.  Then by \cite[Theorem 1]{A}, there are functions with unbounded discrepancy.  Hence all three conditions are necessary.  
  
  Conversely, suppose that all three conditions are satisfied.  Then we can give a basis of generalized eigenfunctions $f_0,f_1,f_2,\dots, f_r,$ where $f_0$ is the Perron Frobenius eigenfunction.  For those that correspond to eigenvalue of modulus less than $1,$ it is easily checked using the path space decomposition that these have bounded discrepancy.  For those with modulus $1,$ we have that their Birkhoff sums remain bounded as they are eigenfunction coboundaries.
\end{proof}


\subsection{ Typical orbits }

So far, we have focused on proving theorems for a fixed point $u$ of $\theta.$  We now show how Theorem~\ref{thm:clt1} changes when instead of looking at $u,$ we look at other sequences $v$ from the orbit closure of $u$ (see Section \ref{sec:dynamics} for the relevant background).

\begin{Theorem}
\label{thm:clt2}
        Let $(X_{\theta},\mathcal{B},\mu,T)$ be the substitution dynamical system arising from the primitive substitution $\theta.$
Let $f$ be a left eigenfunction of $M$ with eigenvalue $\lambda_f$ with $|\lambda_f| = 1$ which is not a coboundary. 
Let $K_N$ be a random variable with uniform distribution on $\left\{ 1,2,\ldots, N \right\}.$  Then for every $v \in X_{\theta},$ there is a sequence $(a_{v,N})_{N=1}^{\infty}$ so that if $\lambda_f \neq 1$
\[
\frac{\functional(v_{\leq K_N})
 - a_{v,N}
}{\sqrt{\log_\lambda(N)}}
\Rightarrow
Z_f
\]
as $N\to\infty$ and if $\lambda_f = 1,$
\[
\frac{\functional(v_{\leq K_N})
 - a_{v,N}
-\log_\lambda(N)\int_{\alphabet} f(a)\sma(da)
}{\sqrt{\log_\lambda(N)}}
\Rightarrow
Z_f.
\]
Further, for $\mu$-almost every $v,$ there is a constant $C>0$ independent of $v$ so that
\begin{equation}
\label{eq:LIL}
\begin{split}
        \limsup_{N \to \infty} 
        \frac{|a_{v,N} - \log_\lambda(N)\int_{\alphabet} f(a)\sma(da)|}{\sqrt{\log N \log \log \log N}}
        =C 
        &\quad \text{for } \lambda_f =1  \\
        \limsup_{N \to \infty} 
        \frac{|a_{v,N} |}{\sqrt{\log N \log \log \log N}}
        =C
        &\quad \text{for } \lambda_f \ne 1.
\end{split}
\end{equation}
\end{Theorem}

Hence the fixed point differs from typical orbits in that $a_{v,N} \equiv 0,$ while all orbits of eigenfunctions of modulus $1$ give central limit theorems.

\subsection{ Examples }
\begin{Example}[$\lambda_f = 1$]
  \label{ex:eigenvalue1} 
Let $\alphabet = \{a, b\}$ and $\theta_1, \theta_2$ and $\theta_3$ are given by
\[\theta_1: a \to aab, \quad b \to bba. \]
\[\theta_2: a \to aab, \quad b \to bab. \]
\[\theta_3: a \to aba, \quad b \to bab. \]
 The $\theta_i$-matrix $M$ is given by  
\[ \left( \begin{array}{ccc}
2 & 1  \\
1 & 2  \end{array} \right).\]
Then $f=[1,-1]$ is a (left) eigenvector of $M$ corresponding to the eigenvalue $\lambda_f=1$.
 Using the information obtained in the Example \ref{ex:main}, we see that
 \begin{enumerate}[(i)]
\item $\theta_1$: It has drift $0$ from equation \eqref{eq:maindrift}. Also $f$ is not a coboundary, otherwise
\[1 = f(\theta(a)_1) = h(\theta(a)_1) - h(\theta(a)_2) = h(a) - h(a) =0,\]
 which is impossible. So we have a central limit theorem.
\item $\theta_2$: Similarly we see that $f$ is not a coboundary. In this case we see a logarithmic drift since 
\[ \int_{\alphabet} f(a) \sma(da) = \sum_{(a,j) \in \statespace} {\hat \sigma}(a,j){\hat \rho}(a,j) \functional[f]( \theta(a)_{<j}) = (0 +1 +2 +0 -1 + 1) \frac{1}{6} = \frac{1}{2}.\]
So we have a central limit theorem with non-zero drift.
\item $\theta_3$: A fixed point of this substitution is periodic: 
\[ababab \cdots.\]
So there is no central limit theorem. 
\end{enumerate}
\end{Example}

\begin{Example}[Rational eigenvalues]
\label{ex:rational}
Let $\alphabet = \{a, b\}$ and $\theta$ is given by
\[\theta: a \to abb, \quad b \to baa. \]
Then the $\theta$-matrix $M$ is 
\[ \left( \begin{array}{ccc}
1 & 2  \\
2 & 1  \end{array} \right)\]
and eigenvalues and corresponding (left) eigenvectors of $M$ are
\[ \lambda =3, \sigma = [1,1], \quad \lambda_f =-1, f = [1,-1].\]
By a similar computation to part (i) of Example~\ref{ex:eigenvalue1}, we see that $f$ is not a coboundary, so we have a CLT.
\end{Example}

\begin{Example}[Irrational eigenvalues]
(Adapted from \cite[Proposition 4.1]{BBH}).
\label{ex:salem}
Let $\alphabet = \{a, b, c, d\}$ and for every $n \geq 1$, $\theta_n$ is given by
\[
\theta_n =
\begin{cases}
a \to ad \\
b \to adbbd \\
c \to ad (bc)^{n+1} bd \\
d \to ad (bc)^n bd.
\end{cases}
\]
Then the characteristic polynomial of the {$\theta_n$-matrix} $M_n$ has the roots $\lambda_1, \lambda_2, \lambda_3, \lambda_4$ such that  
\begin{enumerate}[(i)]
\item $\lambda_4 = 1/ \lambda_1$, $\lambda_3 = 1 / {\lambda_2}$ and $\lambda_1 > |\lambda_2| = |\lambda_3| > \lambda_4$.
\item $\lambda_2$ and $\lambda_3$ are of modulus $1$, so they are Salem numbers.
\item If $\lambda_2 = e^{2 \pi i \alpha}$, then $\alpha$ is irrational.
\end{enumerate}
\end{Example}

In matrix theory, various authors have studied how to determine the $d$-tuples of complex numbers which can occur as the eigenvalues of a primitive matrix. Especially, Boyle and Handelman \cite{BH} formulated ``Spectral Conjecture". Later Kim, Ormes and Roush \cite{KOR} obtained the following result.

For $\Lambda = (\lambda_1, \lambda_2, \dots, \lambda_d)$, denote
\[ tr (\Lambda^n) = \sum_{i=1}^d (\lambda_i)^n \quad \text{and} \quad tr_n(\Lambda) = \sum_{k|n} \mu \left(\frac{n}{k} \right) \, tr (\Lambda^k) \]
where $\mu$ is the M\"obius function.

\begin{Theorem}
Let $\Lambda = (\lambda_1, \lambda_2, \dots, \lambda_d)$ be a $d$-tuple of nonzero complex numbers with $|\lambda_1| \leq |\lambda_2| \leq \cdots \leq |\lambda_d|.$ There exists a primitive integer matrix $A$ such that $\det (I- At) =\prod_{i=1}^d (1 - \lambda_i t)$ if and only if
\begin{enumerate}[(i)]
\item the polynomial $\prod_{i=1}^d (1- \lambda_i t)$ has integer coefficients,
\item $\lambda_1 > |\lambda_i|$ for $i=2,3, \dots ,d$
\item $tr_n(\Lambda) \geq 0$ for all $n \geq 1$.
\end{enumerate}
\end{Theorem}
Note that for a matrix $A$, there is $m \geq 0$ such that $\det (tI - A) = t^m \prod_{i=1}^d (t-\lambda_i)$ if and only if $\det(I - At) = \prod_{i=1}^d (1 - \lambda_i t).$

\begin{Example}[Eigenfunction coboundaries]
\label{ex:coboundary}
Let $\alphabet = \{a, b, c, d\}$ and $\theta$ is given by
\[\theta: a \to ab, \quad b \to ca, \quad c \to cd, \quad d \to ac. \]
Then the $\theta$-matrix $M$ is 
\[ \left( \begin{array}{cccc}
1 & 1 & 0 & 1 \\
1 & 0 & 0 & 0 \\
0 & 1 & 1 & 1 \\
0 & 0 & 1 & 0 
 \end{array} \right)\]
Eigenvalues of $M$ are 
\[\lambda_1 = 2, \lambda_2 = -1, \lambda_3 =1 , \lambda_4 =0\]
and corresponding (left) eigenvectors of $M$ are
\[ \sigma= [1,1, 1, 1], \quad [-1,2,-1,2], \quad [-1,0,1,0] \quad [1,-1,-1,1].\]
 A right Perron-Frobenius eigenvector $\rho$ is given by $\rho^t = [2/6,1/6,2/6,1/6]$ so that $\sum \sigma(a) \rho (a) =1$. 
Then we have that $\sm$ has the form 
\[
(1/6,1/6,1/12,1/12,1/6,1/6,1/12,1/12)
\] 
on the state space $\statespace =\{(a,1), (a,2), (b,1), (b,2), (c,1), (c,2), (d,1), (d,2)\}.$

\begin{enumerate}[(i)]
\item $\lambda_f = 1$ and $f = [-1,0,1,0].$
One has \[ \int_{\alphabet} f(a) \sma(da) = 0\]
and the function $h: \alphabet \to \RR$ such that $h(a) = h(d) =0, h(b) = h(c)=1$ satisfy the condition (iii) and (iv) in Proposition \ref{prop:coboundary}, so $f$ is a coboundary.
\item $\lambda_f = -1, f=[-1,2,-1,2].$
We can obtain \[ \int_{\alphabet} f(a) \sma(da) = -\frac{1}{2}\] and it is not difficult to show that $f$ is not a coboundary by checking condition (iii) in Proposition \ref{prop:coboundary}. 
\end{enumerate}
\end{Example}

\begin{Example}
[Fibonacci substitution]
The Fibonacci substitution $\theta$ is given by 
\[
\theta: a \to ab, \,\, b \to a.
\]
The $\theta$-matrix of the substitution is
\[M= \left( \begin{array}{ccc}
1 & 1  \\
1 & 0  \end{array} \right).\] 
Let $\alpha = \frac{-1 + \sqrt{5}}{2}$. Then $M$ has eigenvalues $\lambda_1 = 1 + \alpha$ and $\lambda_2 = - \alpha$.

Let $u = (u_n)_{n=1}^{\infty} = \lim_{n \rightarrow \infty} \theta^n(a)$. The sequence $u$ also can be obtained by a rotation by ${\alpha}$ (c.f.\,\cite[Proposition 5.4.9]{Fogg}):
\[
u_{n} = a \,\, \text{if } \{n \alpha\} \in [1-\alpha, 1) \quad u_n = b \,\, \text{if } \{n \alpha \} \in [0, 1- \alpha). 
\]
Let $f$ be a function such that $f(a) = 0 - (1-\alpha)$ and $f(b) =1 - (1 - \alpha)$. Then, for $I = [0, 1-\alpha)$, 
\[Z_{\alpha} (n; I) - n(1-\alpha) = \sum_{k=1}^n f(u_k) = S_f (u_{\leq n}).\]
Moreover 
$f = (-1+\alpha, \alpha)$
is an eigenvector corresponding to the eigenvalue $\lambda_2 =  - \alpha$, so $|\lambda_2| <1$.  
For Beck's case, if $I=[0,x)$ for $x \in \Q,$ then we have a central limit theorem for
 $ Z_{\alpha} (n; I) - n(1-\alpha).$
From Theorem \ref{thm:<1}, when $x=1-\alpha,$ the limiting distribution of $ Z_{\alpha} (n; I) - n(1-\alpha)$ is supported on a Cantor set and has a different normalization than when $x \in \Q.$  

\end{Example}

\section{Substitution dynamical systems}
\label{sec:dynamics}
Let $\alphabet$ be a finite set of letters, endowed with the discrete topology, and let $\alphabet^{\ZZ}$ have the product topology, so that $\alphabet^{\ZZ}$ is a compact metric space
and the shift map $T$ given by $(Tu)_n = u_{n+1}$ is a homeomorphism. The pair $(\alphabet^{\ZZ}, T)$ is called the full shift on the alphabet $\alphabet$. If $X$ is a closed $T$-invariant subset of $\sequences$, the pair $(X,T)$ is called a \emph{subshift}.


Given $x \in \alphabet^{\ZZ}$, let $L(x)$ be the set of all finite words appearing in $x$. The language of $\theta$, denoted by $L_{\theta}$, is the set of all finite words occurring in $\theta^n(a)$ for some $n \geq 0$ and $a \in \alphabet$. Let $X_{\theta} = \{ x \in \alphabet^{\ZZ}: L(x) \subset L(\theta) \}$. Then $X_{\theta}$ is closed in $\alphabet^{\ZZ}$ and invariant under the shift. We denote by $T$ the restriction of the shift to $X_{\theta}$. The pair $(X_{\theta}, T)$ is called the \emph{ (two-sided) substitution subshift} associated to $\theta$.
It is known that $(X_{\theta},T)$ is minimal and uniquely ergodic (see \cite{Fogg} or \cite{Q}).

\begin{Remark}
For a given substitution $\theta$, there exist two letters $a, b \in \alphabet$ and $k \in \NN$ such that 
\begin{itemize}
\item $a$ is the last word of $\theta^k(a)$,
\item $b$ is the first word of $\theta^k(b)$,
\item $ab \in L(\theta)$. 
\end{itemize}
Then there exists $v \in \alphabet^{\ZZ}$ such that $v_{-1} = a, v_0 =b$ and $\theta^k(v) =v$. We say that $v$ is a \emph{(two-sided) fixed point} of $\theta$. In this case we have $X_{\theta} = \overline{\{T^n v: n \in \ZZ \}}$.
\end{Remark}

\begin{Remark}
We also can define the \emph{one-sided substitution subshift} associated to $\theta$ by the following. 
By the construction above, there is a fixed point $u$ for $\theta$. Let $\tilde{X}_{\theta}$ be the orbit closure $\overline{\{T^n u : n \in \mathbb{N}_0 \}}$ (where $\mathbb{N}_0 = \mathbb{N} \cup \{0\}$).  The pair $(\tilde{X}_\theta, T)$ is the substitution subshift. This definition can be checked to be independent of the choice of fixed point $u.$ 
The projection $\pi: \alphabet^{\ZZ} \to \alphabet^{\NN}$ maps $X_{\theta}$ onto $\tilde{X}_{\theta}$ and $(X,T, \pi)$ is the natural extension of $(\tilde{X}_{\theta},T)$, that is, for every dynamical system $(Y,S)$ and every factor map $\phi: Y \to \tilde{X}_{\theta}$ there exists a unique factor map $\psi: Y \to X_{\theta}$ with $\pi \circ \psi = \phi$.
\end{Remark}

\subsection{Desubstitution}
Let $\theta$ be a primitive substitution with non-periodic fixed point. The result on recognizability by Moss{\'e} allows one to desubstitute $w$ in $X_{\theta}$ (c.f.\,\cite{Fogg} for details):
\[
w = \cdots w_{-m} \cdots w_{-1} w_0 w_1 \cdots w_n \cdots = \cdots \theta (y_{-1}) \theta (y_0) \theta (y_1) \cdots,
\]
where $w_0$ lies in $\theta(y_0).$

Thus, for a point $w \in X_{\theta}$, there exists a unique sequence $(p_i,c_i,s_i)_{i \in \NN_0} \in (\words \times \alphabet \times \words)^{\NN_0}$ such that $\theta (c_{i+1}) = p_i c_i s_i$ and
\[
w = \cdots \theta^2 (p_2) \theta (p_1) p_0 . c_0 s_0  \theta (s_1) \theta^2 (s_2) \cdots 
\]
which is called a prefix-suffix decomposition of $w$.
If only finitely many $s_i$ are non-empty, then there exist $a \in \alphabet$ and $l, s \in \NN$ such that
\[
x_{[0, \infty)} = c_0 s_0 \theta(s_1) \theta^2 (s_2) \cdots \theta^l(s_l) \lim_{n \rightarrow \infty} \theta^{ns}(a).
\]
Similarly, if only finitely many $p_i$ are non-empty, then there exist $b \in \alphabet$ and $m, t \in \NN$ such that
\[
x_{(-\infty, -1]} =  \lim_{n \rightarrow \infty} \theta^{nt}(b) \theta^m (p_m) \cdots \theta(p_1) p_0.
\]

\subsection{Adic transformations}

As $p \to \infty$, $\champm$ induces the Markov measure $\champm[\infty]$ on the infinite path space $\pathspace[\infty].$ Following Livshits \cite{Livshits}, we define the adic transformation $T_A$ on  $\pathspace[\infty]$ as following: given ${\bf x} = (v_1,k_1)(v_2, k_2) (v_3, k_3) \cdots \in \pathspace[\infty]$,
\begin{itemize}
\item if $k_1 < |\theta(v_2)|$, 
\[ T_A({\bf x}) = (v_1, k_1 +1) (v_2, k_2) (v_3, k_3) \cdots \]
\item otherwise, let $\ell$ be the smallest positive integer such that $k_{\ell} < |\theta(v_{\ell+1})|$, then 
\[ T_A({\bf x}) = (b_1, 1) (b_2,1) \cdots (b_{\ell-1},1) (b_{\ell}, k_{\ell} +1) (v_{\ell+1}, k_{\ell+1}) (v_{\ell+2}, k_{\ell+2}) \cdots \]
where $b_{\ell} = \theta (v_{\ell+1})_{j_{\ell} +1}$ and $b_i = \theta(b_{i+1})_1$ for $1 \leq i \leq \ell -1$.
\end{itemize} 

It is known that the adic transformation is uniquely ergodic and it is measurably isomorphic to the substitution subshift.

Consider the set $C$ of the form 
\begin{equation}
\label{eq:form}
C = [(v_1,k_1) \cdots (v_m, k_m)] \, \backslash \, [(v_1,k_1) \cdots (v_{m-1}, k_{m-1}) (v_m, |\theta(v_{m})|) (v_{m+1}, k_{m +1})] 
\end{equation}
Then for large $N$, $\nu_N(C) = \nu_N (T_A (C))$. Since the sets of the form \eqref{eq:form} generate the Borel $\sigma$-algebra on $\champm[\infty]$, by Proposition \ref{prop:dist} below, $T_A$ is invariant with respect to $\champm[\infty]$. Then $(\pathspace[\infty], \champm[\infty], T_A)$ is measurably isomorphic to the substitution subshift $(X_{\theta}, \mu, T)$.

Also, as $p \to \infty$, $\mpm$ induces the stationary Markov measure $\mpm[\infty]$ on the infinite path space $\pathspace[\infty]$ with the invariant measure $\sm$ given by $\sm(y) = \hat{\sigma}(y) \hat{\rho}(y)$ and the transition probability given by $p(y,z) = \frac{\ssim[y][z][1] \hat{\rho} (z)}{\lambda \hat{\rho}(y) }$, where $\lambda$ is the Perron-Frobenius eigenvalue of $\ssim[\cdot][\cdot][1]$ and $\hat{\sigma}$ and $\hat{\rho}$ are corresponding left and right eigenvectors with $\sum_y \hat{\sigma}(y) \hat{\rho}(y) =1$. It is known \cite{Parry} that $\mpm[\infty]$ is the maximal measure for the topological Markov shift on $\pathspace[\infty]$.

\section{Comparing measures on path space}
\label{sec:comparing}

Recall that for $\bold{x}=x_1 x_2 \cdots x_p  \in \pathspace,$
\[
        \champm\left( \left\{ \bold{x} \right\} \right)
        =
        \cha(\{x_1\})
        \prod_{i=1}^{p-1} \tm(x_{i},x_{i+1}).
\]

We will show both of $\upm$ and $\champm$ are very similar for large $p.$  To compare them, we recall the notion of total variation distance.  For two measures $\mu$ and $\nu$ on a common measure space $(X,\Omega),$ the total variation distance $\dtv$ is given by
\[
        \dtv(\mu,\nu) = \sup_{A \in \Omega} |\mu(A) - \nu(A)|.
\]
The following is now an exercise in \emph{coupling} (see \cite[Chapter 4,5]{LPW} for an introduction): 
\begin{Proposition}
        For any $r < p \in \NN$,
        let $S^r : \pathspace \to \pathspace[p-r]$ be given by
\[
        S^r( x_1x_2\dots x_p )
        =
        x_1x_2\dots x_{p-r}.
\]
        For every $c_1 > 0$ there is a constant $c_2 >0$ such that for all $p \in \NN,$ all $y \in \statespace$ and all integers $r > c_2 \log p,$
        \[
                \dtv(\champm \circ S^{r}, \upm \circ S^{r}) < p^{-c_1}.
        \]
        \label{prop:coupling}
\end{Proposition}

\begin{proof}
We begin by defining coupling of two probability measures on a common probability space. Suppose that $\nu_1$ and $\nu_2$ are probability measures on a probability space $(X, \mathcal{B})$. A coupling of $\nu_1$ and $\nu_2$ is a probability measure $\gamma$ on the product space $(X \times X, \mathcal{B} \otimes \mathcal{B})$ such that marginals are $\nu_1$ and $\nu_2$. 
The total variation norm of $\nu_1$ and $\nu_2$ can be expressed in terms of couplings of $\nu_1$ and $\nu_2$. Specifically by \cite[Proposition 4.7]{LPW}, 
\[
\dtv(\nu_1, \nu_2) = \inf_{\gamma: \text{ couplings}} \gamma \{(x,z) \in X \times X: x \ne z \}.
\]
To bound for total variation norm of  $\champm \circ S^{r}$ and $\upm \circ S^{r}$, it suffices to construct a coupling $\gamma$ such that $\gamma \{(x,y) \in \pathspace \times \pathspace: x \ne y \} < p^{-c_1}$.
%
We recall for convenience \eqref{eq:upm}, which states that for $\bold{x}=x_1 x_2 \cdots x_p,$ 
\begin{align*}
  &\upm (\{\bold{x}\}) = \nonumber\\
        &\frac{\ssim[x_1][y][p-1]}{\ssim[*][y][p-1]}
	h^y_p(x_1,x_2)
	h^y_{p-1}(x_2,x_3)
	\dots
	h^y_{3}(x_{p-2},x_{p-1})
	\frac{\ssim[x_{p-1}][x_p][1]}{\ssim[x_{p-1}][y][p-1]}
	\one[x_p = y].
\end{align*}
For a fixed $x \in \statespace$, $h^y_k(x, \cdot)$ and $\tm(x, \cdot)$ are probability measures. 
For any pair $(x,z) \in \statespace^2$ and any $2 < k \leq p,$ we define a coupling $G_k ((x,z), (\cdot, \cdot))$ of $h^y_k(x, \cdot)$ and $\tm(z, \cdot)$ such that $G_k ((x,z), (\cdot, \cdot))$ attains $\dtv(h^y_k(x, \cdot), \tm (z, \cdot)).$  This coupling can in fact be given explicitly, see \cite[Remark 4.8]{LPW}.
We also let $G_2 ((x,z), (\cdot, \cdot))$ be any coupling of $\delta_y$ and $\tm(z,\cdot),$ and we let $H_p$ be the coupling that attains the total variation distance of the measures
\[
	\frac{\ssim[\cdot][y][p-1]}{\ssim[*][y][p-1]} \quad \text{and} \quad
	\cha(\{\cdot\}).
\]
Hence by Perron-Frobenius theory, we have
\[
	H_p(x \ne z) \ll e^{-cp}.
\]

Now we define the coupling $\gamma$ on $\pathspace \times \pathspace$:
\[
	\gamma(\bold{x}, \bold{z})  = H_p( (x_1,z_1)) \prod_{i=1}^{p-1} G_{p-i+1} ((x_{i}, z_{i}), ( x_{i+1}, z_{i+1})),
\]  
so that $\gamma$ is the law of a Markov chain on $\pathspace \times \pathspace.$

For any measures $\mu_1, \mu_2$ on a countable space $X,$ we have by \cite[Proposition 4.2]{LPW} that $\dtv(\mu_1, \mu_2) = \frac{1}{2} \sum_{x \in X} |\mu_1(x) - \mu_2(x)|.$
Hence for all $k \leq p,$ by \eqref{eq:pconvergence}, 
\[
 \dtv(\tm(a, \cdot), h_{k}^y (a, \cdot)) \leq |\statespace| e^{-ck}.
\]
Then 
\begin{align*}
&\gamma \left( \exists j \leq p - r: x_{j} \ne z_{j} \right) \\
&= \sum_{j=1}^{p-r} \gamma(x_{j} \ne z_{j}, x_{k} = z_{k} \,\, (\forall k < j)) \\
&\leq 
\gamma( x_1 \neq z_1 )
+
\sum_{j=2}^{p-r} \gamma(x_{j} \ne z_{j} | x_{j-1} = z_{j-1})  \\
&
\leq H_p( x_1 \neq z_1 )
+
\sum_{j=2}^{p-r} \sup_{x \in \statespace} \dtv(\tm(x, \cdot), h_{p -j +2}(x, \cdot)) \\
& \ll e^{-cr}.
\end{align*}
Taking $r = C\log p$ for sufficiently large $C$ completes the proof.
\end{proof}

We can also compare $\mpm$ and $\champm$ in a similar way, which is a standard result on primitive Markov chains.
\begin{Proposition}
        For any $r < p \in \NN$,
        let $L^r : \pathspace \to \pathspace[p-r]$ be given by
\[
        L^r( x_1x_2\dots x_p )
        =
	x_{r+1}x_{r+2}\dots x_{p}.
\]
        For every $c_1 > 0$ there is a constant $c_2 >0$ such that for all $p \in \NN$ and all integers $r > c_2 \log p,$
        \[
                \dtv(\champm \circ L^{r}, \mpm \circ L^{r}) < p^{-c_1}.
        \]
        \label{prop:mpmcoupling}
\end{Proposition}
\noindent For a proof, see \cite[(5.2)]{LPW}.

As a consequence of Proposition~\ref{prop:coupling} and Proposition~\ref{prop:decomposition}, we have that 
\begin{Proposition}
\label{prop:dist}
For every $c_1 > 0,$ there is a constant $c_2 >0$ so that for all $p \in \NN,$ $a \in \alphabet,$ 
$N \in \NN$ with $|\theta^{p-1}(a)| \leq N < |\theta^p(a)|,$
and all integers $r > c_2 \log p,$
        \[
                \dtv(\champm \circ S^{r}, \nu_N \circ S^{r}) \ll p^{-c_1}.
        \]
        \label{prop:Nmpm}
\end{Proposition}
\begin{proof}

 By Proposition~\ref{prop:decomposition},
 \[
        \nu_N \circ S^r
                =
                \sum_{\ell=1}^p \sum_{1 \leq j < k_\ell} 
                \frac{\ssim[*][(v_\ell,j)][\ell]}{N} 
                \upm[(v_\ell,j)][\ell] \circ S^{(r-(p-l))_+}.
 \]
We also have that
 \[
        \champm \circ S^r
                =
                \sum_{\ell=1}^p \sum_{1 \leq j < k_\ell} 
                \frac{\ssim[*][(v_\ell,j)][\ell]}{N} 
                \champm \circ S^{r}.               
 \]
 Note that by stationarity,
 $\champm \circ S^{r} =\champm[p-r]$.

 For any event $A$ and large $r$ as in the Proposition~\ref{prop:coupling}, let $r_0 = [r/2]$
\begin{align*}
        &|\nu_N \circ S^r (A) - \champm \circ S^r (A)|\\
        &\leq
        \sum_{\ell=1}^{p-r_0} \sum_{1 \leq j < k_\ell} 
        \frac{\ssim[*][(v_\ell,j)][\ell]}{N} \\
        &+ \sum_{\ell=p-r_0}^p \sum_{1 \leq j < k_\ell} 
        \frac{\ssim[*][(v_\ell,j)][\ell]}{N}
        |\upm[(v_\ell,j)][\ell] \circ S^{r-(p-l)}(A) - \champm[l] \circ S^{r-(p-l)}(A)|\\
        &\leq \frac{|\theta^{p-r_0}(a)|}{N} + (p-r_0)^{-c_1},
\end{align*}
where we have applied Proposition~\ref{prop:coupling} to the third line.
Then, for some $\alpha>0$,
\[ 
\frac{|\theta^{p-r_0}(a)|}{N} 
\ll \frac{1}{|\lambda^{r_0-1}|} \ll p^{-\alpha c_2}.
\]
Picking $c_2$ sufficiently large, the result follows. 
\end{proof}

\section{Proofs for $|\lambda_f|=1$}

Theorem~\ref{thm:clt1} will follow immediately from Proposition~\ref{Prop:service} combined with Proposition~\ref{prop:coboundary}.  However, Proposition~\ref{prop:coboundary} relies on Proposition~\ref{Prop:service}, so we present Proposition~\ref{Prop:service} first.

\begin{Proposition}
        \label{Prop:service}
Suppose that $f$ is a left eigenfunction of $M$ with eigenvalue $\lambda_f$ with $|\lambda_f| = 1$ for which there is no function $h : \alphabet \to \C$ so that
for all $c \in \alphabet$
        \[
        \functional[f](\theta(c)_{<j})
        =
        \int_{\alphabet} f(b)\sma(db)
        +h(c) - \lambda_f^{-1} h(\theta(c)_j).
        \]
Let $K_N$ be a random variable with uniform distribution on $\left\{ 1,2,\ldots, N \right\},$ and let $a \in \alphabet$ be fixed.  Then if $\lambda_f \neq 1,$ as $N\to\infty$
\[
	\sup_{\substack{ \ell \in \NN \\
	|\theta^\ell(a)| \geq N}}
	\dbl\left( 
\frac{\functional(\theta^\ell(a)_{\leq K_N})
}{\sqrt{\log_\lambda(N)}}
,Z_f\right)
\to 0.
\]
If $\lambda_f = 1,$ as $N\to\infty$
\[
	\sup_{\substack{ \ell \in \NN \\
	|\theta^\ell(a)| \geq N}}
	\dbl\left( 
\frac{\functional(\theta^\ell(a)_{\leq K_N})
-
\log_\lambda(N)
\int_{\alphabet} f(b)\sma(db)
}{\sqrt{\log_\lambda(N)}}
,Z_f\right)
\to 0.
\]
\end{Proposition}

\begin{proof} 
Given $N$, one can find $p_N$ such that $|\theta^{p_N-1}(a)| < N \leq |\theta^{p_N}(a)|$.
Now we define a function $\check{f}$ on $\statespace$ by
$\check{f}(v_i,k_i) = \functional(\theta(v_i)_{< k_i}).$
Then we have for any $1 \leq K \leq N,$ that
\[
        \lambda_f \functional (\theta^\ell(a)_{<K}) = \sum_{i=1}^{\ell} \lambda_f^{i} \check{f}((v_i,k_i)),
\]
where 
\(
\Psi_{a,\ell}(K) = (v_1,k_1)(v_2,k_2) \dots (v_{\ell},k_{\ell})
.\)
For $m > p_N,$ we have that $k_m = 1.$  Hence $\check{f}(v_m,k_m) = 0$ for all these $m,$ and we have
\[
        \lambda_f \functional (\theta^\ell(a)_{<K}) = \sum_{i=1}^{p_N} \lambda_f^{i} \check{f}((v_i,k_i)),
\]

We will show a central limit theorem for 
\(
\frac{\lambda_f S_f(\theta^\ell(a)_{<K_N}) - \sum_{i=1}^{p_N} \lambda_f^i\int \check{f} d \sm}{\sqrt{\log_\lambda N}}.
\)
The desired central limit theorems follow immediately from this.

Let $Z = \frac{1}{\sqrt{\log_{\lambda}N}} \sum_{i=1}^{p_N} \lambda_f^i( \check{f}((v_i,k_i)) - \int \check{f} d \sm) $, which is a function on $\pathspace$.  Then we have that
\[
  \frac{\lambda_f \functional (\theta^\ell(a)_{<K})
  -
   \sum_{i=1}^{p_N} \lambda_f^i\int \check{f} d \sm 
  }{\sqrt{\log_\lambda(N)}}
  = Z(\Psi_{a,p_N}(K)).
\]
By definition of $\nu_N,$ we therefore have that for any bounded Lipschitz function $\phi,$
\begin{align*}
  \frac{1}{N}\sum_{K=1}^N\phi\bigl( \frac{\lambda_f \functional (\theta^\ell(a)_{<K})
  -\sum_{i=1}^{p_N} \lambda_f^i\int \check{f} d \sm 
  }{\sqrt{\log_\lambda(N)}} \bigr)
  &=\Exp \phi\bigl( \frac{\lambda_f \functional (\theta^\ell(a)_{<K_N})
  -\sum_{i=1}^{p_N} \lambda_f^i\int \check{f} d \sm 
  }{\sqrt{\log_\lambda(N)}} \bigr) \\
  &=\Exp \phi\bigl( Z(\Psi_{a,p_N}(K_N)) \bigr) \\
  &=\int_{\pathspace} \phi(Z(\bx)) \nu_N(d\bx).
\end{align*}
We will show that for any bounded Lipschitz function $\phi$ 
\[
        \int_{\pathspace} \phi(Z(\bx))\nu_N(d\bx) \to \int_{\C} \phi(x) \Phi(dx),  
\]
where $\Phi$ is the probability measure given by $\Phi(A) = \Pr\left[ Z_f \in A \right].$

Let $r = [(\log p_N)^2]$ and write $Z= X+Y$, where
\begin{align*}
 Y &= \frac{1}{\sqrt{\log_{\lambda}N}} \sum_{i=r+1}^{p_N-r} \lambda_f^i( \check{f}((v_i,k_i)) - \int \check{f} d \sm).
\end{align*}
Hence $\|X\|_{\infty} \ll \frac{(\log\log N)^2}{\sqrt{\log N}}.$
Then
\begin{align*}
\int \phi(Z) d \nu_N &= \int \phi(Y) d \nu_N + \int \phi(Z) - \phi(Y) d \nu_N. \\
\intertext{As $\phi$ is Lipschitz, the second integral is at most $\|\phi\|_{\text{Lip}} \|X\|_\infty = o(1).$ Hence}
\int \phi(Z) d \nu_N 
&= \int \phi(Y \circ S^r \circ L^r) d (\nu_N \circ S^r \circ L^r) + o(1)\\
&= \int \phi(Y \circ S^r \circ L^r) d \mpm[p_N] \circ S^r \circ L^r \\
&\hspace{1cm}+ O(\dtv (\nu_N \circ S^r \circ L^r, \mpm[p_N] \circ S^r \circ L^r)) + o(1).\\
\intertext{By Proposition \ref{prop:Nmpm} and Proposition \ref{prop:mpmcoupling}, $ \dtv (\nu_N \circ S^r \circ L^r, \mpm[p_N] \circ S^r \circ L^r) \to 0$ as $N \to \infty.$ Therefore
}
\int \phi(Z) d \nu_N 
&= \int \phi(Y \circ S^r \circ L^r) d \mpm[p_N - 2r] +o(1).\\
\intertext{     Again using that $X$ is uniformly small and the Lipschitzness of $\phi,$ we conclude}
\int \phi(Z) d \nu_N 
&= \int \phi(Z) d \mpm[p_N] +o(1).
\end{align*}
All said, we have shown that
\[
	\dbl\left( 
	\nu_N \circ Z^{-1}, \mpm[p_N] \circ Z^{-1}
	\right) \to 0,
\]
as $N \to \infty$ uniformly in $\ell.$

Now the theorem follows from Theorem \ref{thm:Markov} and the observation that $\frac{p_N}{\log_{\lambda}N} \to 1$, provided that we show that there is no $\check h$  satisfying $P^*P \check h = \check h$ so that $\check f = \int \check f(x)\sm(dx) + \check h - \lambda_f P\check h$, where $\tm^*$ and $\tm$ are defined in \eqref{eq:pstar} and \eqref{eq:p} and
$$(Ph)(x) = \sum_{x \in \statespace} \tm(x,y)h(y) \quad \text{and} \quad (P^*h)(x) = \sum_{x \in \statespace} \tm^*(x,y)h(y).$$

Suppose that there were such an $\check h.$ Then by Theorem \ref{thm:Markov}, we have that
\[
        W=
\sum_{i=1}^{p_N} \lambda_f^i \left( \check{f}((v_i,k_i)) - \int \check{f} d \sm \right).
\]
is uniformly bounded in $N$ for $\mpm[p_N]$-almost every path.  As $\mpm[p]$ has full support on $\pathspace,$ for every $p,$ we have that there is a $C$ so that 
\[
        \sup_{N > 0} \sup_{(v_i,k_i)_{i=1}^{p_N} \in \pathspace[p_N]}
        \left|
        W\left( 
        (v_i,k_i)_{i=1}^{p_N}
\right)\right| < C. 
\]
Hence we also have that
\[
        \sup_{N > 0} \sup_{(v_i,k_i)_{i=1}^{p_N} \in \pathspace[p_N]}
        \left|
        \lambda_f^{-p_N-1}
        W\left( 
        (v_i,k_i)_{i=1}^{p_N}
\right)\right| < C. 
\]
Observe that 
\[
        \lambda_f^{-p_N-1}W=
        \sum_{i=1}^{p_N} \lambda_f^{-(p_N-i+1)}( \check{f}((v_i,k_i)) - \int \check{f} d \sm).
\]
Hence by Theorem \ref{thm:Markov} applied to the reversed chain with transition matrix $\tm^*,$ there must be an $\hat h$ so that $PP^* \hat h = \hat h$ and 
\(
\check f = 
\int \check{f} d \sm
+\hat h
-\lambda_f^{-1} P^* \hat h.
\)

We now turn to characterizing those $\hat h$ for which $PP^*\hat h = \hat h.$  First, we evaluate $PP^*:$ by equations \eqref{eq:pstar} and \eqref{eq:p}
\begin{align*}
        \tm \tm^*\left( (a,j),(b,k) \right)
        &=
        \sum_{(c,\ell) \in \statespace}
        \tm\left( (a,j),(c,\ell) \right)
        \tm^*\left( (c,\ell),(b,k) \right) \\
        &=
        \sum_{(c,\ell) \in \statespace}
        \frac{\one[ \theta(c)_{\ell} = a] 
        {\hat \rho}(c,\ell)}
        {\lambda {\hat \rho}(a,j) }
        \frac{\one[ \theta(c)_{\ell} = b] 
        {\hat \sigma}(b,k)}
        {\lambda {\hat \sigma}(c,\ell) }. 
\end{align*}
Observe that this is nonzero if and only if $a=b.$  Hence $\tm \tm^*$ is a block matrix, with $|\alphabet|$ many blocks, each of which is positive.  Moreover, we have that $\tm \tm^*$ is right stochastic, as it is the product of two right stochastic matrices, and hence 
\[
1
=\sum_{(b,k) \in \statespace}
        \tm \tm^*\left( (a,j),(b,k) \right)
        =\sum_{k=1}^{|\theta(a)|}
        \tm \tm^*\left( (a,j),(a,k) \right).
\]
Thus each of these blocks is itself right stochastic.  As a consequence, the eigenspace of $\tm \tm^*$ with eigenvalue $1$ has dimension $|\alphabet|$ and is spanned by those functions $(a,j) \mapsto h(a),$ where $h : \alphabet \to \C$ is any function.

Hence we have that $\hat h(a,j) = h(a)$ for some function $h.$  Evaluating $P^*\hat h,$ we have
\begin{align*}
        P^*\hat h(a,j)
        &= \sum_{(b,k) \in \statespace}
        \tm^*\left( (a,j),(b,k) \right) h(b) \\
        &= \sum_{(b,k) \in \statespace}
        \frac{\one[ \theta(a)_{j} = b] 
        {\hat \sigma}(b,k)}
        {\lambda {\hat \sigma}(a,j) }
         h(b) \\
&= 
 h(\theta(a)_j)
\sum_{(b,k) \in \statespace}
        \frac{\one[ \theta(a)_{j} = b] 
        {\hat \sigma}(b,k)}
        {\lambda {\hat \sigma}(a,j) }\\
&= 
 h(\theta(a)_j)
\sum_{(b,k) \in \statespace}
        \tm^*\left( (a,j),(b,k) \right) \\
        &=h(\theta(a)_j).
\end{align*}
Therefore, we have that 
\[
        \check f(a,j) = \int \check f(x)\sm(dx) + h(a) - \lambda_f^{-1} h(\theta(a)_j),
\]
which contradicts the hypothesis on $f$ that we assumed in the statement of the proposition.
\end{proof}

\section{Proof of coboundary proposition}
\label{sec:coboundary}

\begin{proof}[Proof of Proposition~\ref{prop:coboundary}]
        \noindent $(i) \Leftrightarrow (ii): $ This follows from that $(X_{\theta},T)$ is minimal.

        \noindent $(iii) \Leftrightarrow (iv): $ This follows from the identities $f(\theta(a)_j) = \functional[f](\theta(a)_{<j+1}) - \functional[f](\theta(a)_{<j})$ 
and $\functional[f](\theta(a)_{<1}) =0.$

\noindent $(i) \Rightarrow (iii): $ By Proposition \ref{Prop:service}, if $(iii)$ does not hold, then for a fixed point $u = u_1u_2\dots,$ $\limsup\limits_{N \to \infty} | \functional[f](u_{\leq N}) | = \infty$.

\noindent $(iii) \Rightarrow (ii): $ Let $u$ be a fixed point.
Then by \eqref{eq:11}, for each $N$ there exists $(v_1, k_1), \dots, (v_{p_N},k_{p_N}) \in \statespace $ so that 
\[
\functional[f](u_{< N}) = \sum_{i=1}^{p_N} \lambda_f^{i-1} \functional[f](\theta(v_i)_{<k_i}). 
\]
Then, since $v_i = \theta(v_{i+1})_{k_i+1},$ 
\begin{align*}
\functional[f](u_{<N}) 
        &= \sum_{i=1}^{p_N} \lambda_f^{i-1} 
        \left[ \int_{\alphabet} f(c)\sma(dc) + h(v_i) - \lambda_f^{-1}h(\theta(v_i)_{k_i})\right] \\
	&= \sum_{i=1}^{p_N} \left[ \lambda_f^{i-1} \int_{\alphabet} f(c)\sma(dc) \right] - \lambda_f^{-1} h(\theta(v_1)_{k_1}) + \lambda_f^{p_N-1} h(v_{p_N}).
        \end{align*}
If $\lambda_f \ne 1,$ then we are done. If $\lambda_f =1$, one can find $a \in \alphabet$ and $k \in \N$ so that $\theta^k(a)_1 = a$. Then for $0 \leq i \leq k-1,$ 
\[
 \int_{\alphabet} f(c)\,\sma(dc)
 + h(\theta^i(a)_1) - h(\theta^{i+1}(a)_1) = 0.
\]
Adding these sums,
we have $ \int f(x)\,\sma(dx)=0.$ Thus $\functional[f](u_{\leq N})$ is bounded.
\end{proof}


\section{Proof of Theorem \ref{thm:clt2}}

Let $u(a,k) = \theta^k(a)$ for $a \in \alphabet$ and $k, N \in \N$. For $n \leq |u(a,k)|$, define
\[
Z_n^N(a,k) =
\begin{cases}
S_f(u(a,k)_{\leq n}) - \log_{\lambda}(N) \int_{\alphabet} f(b) \sma(db), & \text{if } \lambda_f =1 \\
S_f(u(a,k)_{\leq n}), & \text{otherwise}.
\end{cases}
\]

\begin{Lemma}
\label{lem:uniform}
Let $f$ be a left eigenfunction of $M$ with eigenvalue $\lambda_f$ having $|\lambda_f| =1$ which is not a coboundary.
For any Borel measurable set $A \subset \C$ (or $\R$ in the case $\lambda_f$ and $f$ are real)
let $\Phi(A) = \Pr( Z_f \in A).$
 For any Borel-measurable set $A \subset \C$ with $\Phi(\partial A)=0,$
\[
\max_{\substack{a \in \alphabet \\ k \in \N :|u(a,k)| \geq N}} 
\left|
\frac{1}{N} |\{1 \leq k \leq N: \frac{Z_n^N (a,k)}{\sqrt{\log_{\lambda} N}} \in A \}| - \Phi(A)
\right|
\rightarrow 0 \,\,\, \text{as} \,\, N \rightarrow \infty.
\]
\end{Lemma}
\begin{proof}
	As $\alphabet$ is finite, this follows directly from Proposition~\ref{Prop:service}.	
\end{proof}


Let $M= |u(a,k)|$. For $n \leq M$, define
\[
Y_n^N(a,k) =
\begin{cases}
-S_f(u(a,k)_{[M-n+1, M]}) - \log_{\lambda}(N) \int_{\alphabet} f(b) \sma(db), & \text{if } \lambda_f =1 \\
-S_f(u(a,k)_{[M-n+1,M]}), & \text{otherwise}.
\end{cases}
\]

\begin{Lemma}
\label{lem:reverse}

 For any Borel-measurable set $A \subset \C$ with $\Phi(\partial A)=0,$
 as $N \rightarrow \infty$
\begin{equation}
\label{eq:reverse}
\max_{\substack{a \in \alphabet \\ k \in \N : |u(a,k)| \geq N}} 
\left|
\frac{1}{N} |\{1 \leq n \leq N: \frac{Z_n^N (a,k)}{\sqrt{\log_{\lambda} N}} \in A \}| - \frac{1}{N} |\{1 \leq n \leq N: \frac{Y_n^N (a,k)}{\sqrt{\log_{\lambda} N}} \in A \}|
 \right| 
\rightarrow 0.
\end{equation}
\end{Lemma}
\begin{proof}
Let $M = |\theta^k(a)|$. Note that
\begin{align*}
-S_f(u(a,k)_{[M-n+1, M]}) &= - S_f (u(a,k)) + S_f(u(a,k)_{[1, M-n]}) \\
                                           &= - \lambda_f^k f(a) + S_f(u(a,k)_{[1, M-n]}). 
\end{align*} 
Then
\[ 
\{1 \leq k \leq M: \frac{Y_n^M (a,k)}{\sqrt{\log_{\lambda} M}} \in A \} = \{1 \leq k \leq M: \frac{Z_n^M (a,k)}{\sqrt{\log_{\lambda} M}} \in A + \frac{\lambda_f^k f(a)}{\sqrt{\log_{\lambda} M}} \}
\]
Thus, \eqref{eq:reverse} holds along the subsequence $N_k= |\theta^k(a)| \uparrow \infty$. 
Given a substitution $\theta$, we can define a reverse substitution $\tilde{\theta}$ by $\tilde{\theta}(a) = a_n a_{n-1} \cdots a_1$ for $\theta(a) = a_1 a_2 \cdots a_n$. Since \eqref{eq:reverse} holds along the subsequence $N_k$, if there exists a central limit theorem with a drift for $\theta$ with eigenfunction $f$ and there exists a central limit theorem for $\tilde{\theta}$ and $-f$, then one has the same drift for $\tilde{\theta}$ and $-f$.
Now we can see that $-S_f(u(a,k)_{[M-n+1, M]})$ is a Birkhoff sum of $-f$ on the substitution system associated to $\tilde{\theta}$. 
By Lemma \ref{lem:uniform}, for some $\Phi_1$ and $\Phi_2$,
\[
\max_{\substack{a \in \alphabet \\ k \in \N: |u(a,k)| \geq N}} \left| \frac{1}{N} |\{1 \leq k \leq N: \frac{Z_n^N (a,k)}{\sqrt{\log_{\lambda} N}} \in A \}| - \Phi_1(A) \right| \rightarrow 0,
\]
\[
\max_{\substack{a \in \alphabet \\ k \in \N: |u(a,k)| \geq N}} \left| \frac{1}{N} |\{1 \leq k \leq N: \frac{Y_n^N (a,k)}{\sqrt{\log_{\lambda} N}} \in A \}| - \Phi_2(A) \right| \rightarrow 0.
\]
Since \eqref{eq:reverse} holds along the subsequence $N_k= |\theta^k(a)| \uparrow \infty$, $\Phi_1 = \Phi_2$, 
so the proof is completed.
\end{proof}

\begin{proof}[Proof of Theorem \ref{thm:clt2}]
We will prove for the case $\lambda_f =1$. The proof for $\lambda_f \ne 1$ is analogous.
Let us consider prefix-suffix decomposition of $v$:
\begin{equation}
\label{eq:prefix-suffix}
v = \cdots \theta^2(p_2) \theta(p_1) p_0. c_0 s_0 \theta(s_1) \theta^2(s_2) \cdots 
\end{equation}

\noindent Case I. If only finitely many $s_i$ are non-empty, then for some $a \in \alphabet$ and $l, k \in \N$,
$v_{[0, \infty)} = c_0 s_0 \theta(s_1) \cdots \theta^l(s_l) u$ and $u = \lim_{n \rightarrow \infty} \theta^{kn} (a)$. 
For sufficiently large $n$, $S_f(v_{\leq n}) = S_f(u_{\leq n}) + O(1)$. Thus the result follows from Theorem \ref{thm:clt1}.

\noindent Case II. Otherwise, define $N_l = |c_0 s_0 \theta(s_1) \cdots \theta^l(s_l)|$. Note that $N_{l} \uparrow \infty$. For any positive integer $N$, we choose $N_{\ell}$ such that $N_l \leq N < N_{\ell+1}$. Then set $a_{v,N} = S_f (v_{[1,N_{\ell}]})$. 
Note that 
\[
S_f(v_{[1,n]}) =
\begin{cases}
S_f(v_{[1, N_{\ell}]}) - S_f(v_{(n,N_{\ell}]}), & \text{if } n \leq N_{\ell} \\
S_f(v_{[1, N_{\ell}]}) + S_f(v_{(N_{\ell},n]}), & \text{if } n > N_l.
\end{cases}
\]
Then
\begin{align}
\label{eq:10}
&\frac{1}{N} \left| \{ 1 \leq n \leq N: \frac{S_f (v_{[1,n]}) - a_{v,N} - \log_{\lambda} (N) \int_{\alphabet} f(b) \sma (db)}{\sqrt{\log_{\lambda} N}} \in A \}\right|  \\
&= \frac{N_{\ell}}{N} \frac{1}{N_{\ell}} \left| \{1 \leq n \leq N_{\ell}: \frac{-S_f(v_{(n, N_{\ell}]}) - \log_{\lambda} (N) \int_{\alphabet} f(b) \sma (db) }{\sqrt{\log_{\lambda}N}} \in A \} \right| \notag \\
&+ \frac{N- N_{\ell}}{N} \frac{1}{N- N_{\ell}} \left| \{N_{\ell} < n \leq N: \frac{S_f(v_{(N_{\ell}, n]}) - \log_{\lambda} (N) \int_{\alphabet} f(b) \sma (db) }{\sqrt{\log_{\lambda}N}} \in A \} \right|. \notag
\end{align}
Fix $\epsilon >0$ small so that if $M$ satisfies that $\epsilon N \leq M \leq N$, then $\frac{\log_{\lambda} N}{\log_{\lambda} M} = 1 + O(\epsilon)$ and $\log_{\lambda} \frac{M}{N} = 1 + O(\epsilon)$. 

For any $N$ we have the following three cases:

Case (i): $\epsilon N \leq N_{\ell} \leq N$ and $\epsilon N \leq N- N_{\ell} \leq N$.
If $N$ is large enough so that $N_{\ell}$ is large, 
\begin{equation*}
\begin{split}
&\frac{-S_f(v_{(n, N_{\ell}]}) - \log_{\lambda} (N) \int_{\alphabet} f(b) \sma (db) }{\sqrt{\log_{\lambda}N}} \\
 &= \frac{-S_f(v_{(n, N_{\ell}]}) - \log_{\lambda} (N_{\ell}) \int_{\alphabet} f(b) \sma (db) }{\sqrt{\log_{\lambda}N_{\ell}}} (1+ O(\epsilon))+ o_N(1).
\end{split}
\end{equation*}
So, from Lemma \ref{lem:reverse}
\[
 \frac{1}{N_l} \left| \{1 \leq n \leq N_{\ell}: \frac{-S_f(v_{(n, N_{\ell}]}) - \log_{\lambda} (N) \int_{\alphabet} f(b) \sma (db) }{\sqrt{\log_{\lambda}N}} \in A \} \right|
\]
is as close as to $\Phi(A)$ for small $\epsilon$ and large $N$.
Similarly, 
\[ \frac{1}{N- N_{\ell}} \left| \{N_{\ell} < n \leq N: \frac{S_f(v_{(N_{\ell}, n]}) - \log_{\lambda} (N) \int_{\alphabet} f(b) \sma (db) }{\sqrt{\log_{\lambda}N}} \in A \} \right|\]
is also as close as to $\Phi(A)$. Thus \eqref{eq:10} is close to $\Phi(A)$. 

Case (ii) $N_{\ell} < \epsilon N$. Then 
\[
\frac{N_{\ell}}{N} \frac{1}{N_{\ell}} \left| \{1 \leq n \leq N_{\ell}: \frac{-S_f(v_{(n, N_{\ell}]}) - \log_{\lambda} (N) \int_{\alphabet} f(b) \sma (db) }{\sqrt{\log_{\lambda}N}} \in A \} \right|  < \epsilon.
\]
Also $N -N_{\ell} > (1- \epsilon)N \geq \epsilon N$. Using the same argument as in Case (i), one can see that \eqref{eq:10} is close to $\Phi(A)$.

Case (iii) $N - N_{\ell} < \epsilon N$. It is similar to Case (ii).

Now it remains to show equation \eqref{eq:LIL}. For almost every $v$, the prefix-suffix decomposition \eqref{eq:prefix-suffix} satisfies Case II above, thus $a_{v,N}$ is given by $a_{v,N} = S_f (v_{[1,N_{\ell}]})$ as above. Moreover we claim that there exists $C_0$ such that for almost every $v$,
\begin{equation}
	\label{eq:ell}
 \limsup_{\ell \to \infty} \frac{\log_{\lambda} \frac{N_{\ell +1}}{N_{\ell}}}{\log_{\lambda} \ell} < C_0.
\end{equation}
Note that for some $t>0$, there exists $c >0$, denoting $m = t/c \in \N$, such that if $\log_{\lambda} \frac{N_{\ell +1}}{  N_{\ell}} > t$, then $s_{\ell - m}, s_{\ell - m+1}, \dots, s_l$ are empty-word.
Since the sequences $(p_i,c_i, s_i)_{i \in \NN_0}$ are primitive homogeneous Markov chains, for some $\alpha > 0$,
\[
\Pr \left[ \log_{\lambda} \frac{N_{\ell +1}}{ N_{\ell}} > t \right] < e^{- \alpha t}. 
\]
Choosing $t = \frac{2}{\alpha} \log \ell$ and applying Borel-Cantelli Lemma, the claim follows. 

Now we see that 
\[
	\functional(v_{[1,N_{\ell}]}) = 
	\functional (c_0) +  \functional (s_0) + \lambda_f \functional (s_1) + \cdots +  \lambda_f^{\ell} \functional (s_{\ell}).
\]
As $\functional(v_{[1,N_{\ell}]})$ is an additive functional of the finite state Markov chain $(p_i,c_i, s_i)_{i \in \NN_0}$, the law of iterated logarithm holds (see Theorem~\ref{thm:Markov}).  If $\lambda_f=1,$ this implies there is some $C>0$
\[
\limsup_{\ell \to \infty} 
\frac{|\functional (v_{[1,N_{\ell}}) - \ell \int_{\alphabet} f(a)\sma(da)|}{\sqrt{ \ell \log \log \ell}}
        =C. 
\]
From \eqref{eq:ell}, the desired conclusion holds. For general $\lambda_f$ with $|\lambda_f|=1$, a similar argument completes the proof. 

\end{proof}

\section{Proofs for $|\lambda_f| \neq 1$}
We begin with the proof of Theorem~\ref{thm:<1}.
\begin{proof} [Proof of Theorem~\ref{thm:<1}]
Given $N$, one can find $p_N$ such that $|\theta^{p_N-1}(a)| < N \leq |\theta^{p_N}(a)|$.
We again define a function $\check{f}$ on $\statespace$ by
$\check{f}(v_i,k_i) = \functional(\theta(v_i)_{< k_i}).$
Then we again have that for any $1 \leq K \leq N,$ that
\[
        \lambda_f \functional (u_{<K}) = \sum_{i=1}^{p_N} \lambda_f^{i} \check{f}((v_i,k_i)),
\]
where $\Psi_{a,p_N}(K) = (v_1,k_1)(v_2,k_2)\dots(v_{p_N},k_{p_N}).$

For any natural number $p,$ let $Z_{p} =\sum_{i=1}^{p} \lambda_f^i \check{f}((v_i,k_i))$, which is a function on $\pathspace$.  We naturally embed $\pathspace$ as the initial coordinates of $\pathspace[\infty],$ and thus also consider $Z_p$ a function on $\pathspace[\infty].$  Also let  $Z_\infty = W_f = \sum_{i=1}^{\infty} \lambda_f^i \check{f}((v_i,k_i)),$ a function on $\pathspace[\infty].$  Observe that
\[
\sup_{p \geq 1} \sup_{\bold{x} \in \pathspace[\infty]} |Z_p(\bold{x})| < \infty.
\]
Moreover, we have that uniformly in $p,$ for all $n > p$
we have that 
\begin{equation}
  \label{eq:Zconvergence}
\sup_{\bold{x} \in \pathspace[\infty]} |Z_p(\bold{x}) - Z_n(\bold{x})| \ll |\lambda_f|^{p}.
\end{equation}
Hence the same estimate holds for the difference of $Z_p$ and $Z_\infty.$  

We will show that for any bounded uniformly continuous function $\phi$ 
\[
  \int_{\pathspace} \phi(Z_{p_N}(\bx)) \nu_N(d\bx) \to \int_{\C} \phi(Z_\infty(\bold{x})) \champm[\infty](d\bold{x}),
\]
which will complete the proof.  Exactly as in the proof of Proposition~\ref{Prop:service}, using \eqref{eq:Zconvergence} and Proposition \ref{prop:Nmpm}, we have that
\[
  \int \phi(Z_{p_N}) d \nu_N 
  = \int \phi(Z_{p_N}) d \champm[p_N] +o(1)
  = \int \phi(Z_{p_N}) d \champm[\infty] +o(1)
\]
But by uniform continuity of $\phi$ and \eqref{eq:Zconvergence}, we have that
\[
  \int \phi(Z_{p_N}) d \champm[\infty]
  =\int \phi(Z_{\infty}) d \champm[\infty] +o(1),
\]
so the proof is complete.

\end{proof}

\begin{proof}[Proof of Theorem~\ref{thm:>1}]
  Let $p= p(\ell) \in \NN$ be such that $|\theta^{p-1}(a)| < N_\ell \leq |\theta^p(a)|.$
  For every $\bold{x}=x_1x_2\cdots x_p \in \pathspace,$ we have by Proposition~\ref{prop:decomposition} that
  \begin{equation}
    \label{eq:>1nu}
	  \nu_{N_\ell}
                \left( \left\{ \bold{x} \right\} \right)
                =
		\sum_{q=1}^{p} \sum_{1 \leq j < k_q} 
                \frac{\ssim[*][(v_q,j)][p-q+1]}{N_\ell} 
                \upm[(v_q,j)][p-q+1]
		\left( \left\{ x_1x_2\cdots x_{p-q} \right\} \right),
	      \end{equation}
	(Caution: we have used the reversed path here),
	where $\Psi_a^r(N_\ell) = (v_1(\ell),k_1(\ell))(v_2(\ell),k_2(\ell))\cdots.$

	We will begin by showing that as a measure on $\statespace^{\infty},$ $\nu_{N_\ell}$ converges.  Let ${\tilde \nu}_{N_\ell}$ be a measure on $\statespace^{\infty}$ 
	given by the property that for any $1 \leq K < N_\ell,$
	\( 
	{\tilde \nu}_{N_\ell}( \Psi_a^r(K))
	=\nu_{N_\ell}( \Psi_{a,p}(K)).
	\)
	It follows that for any cylinder $[\bold{x}] = [x_1x_2\dots x_p]$
	\begin{equation}
	  {\tilde \nu}_{N_\ell}( [\bold{x}])
	  = \nu_{N_\ell}(x_px_{p-1}\dots x_1).
	  \label{eq:nur}
	\end{equation}

	We will show that 
	\(
	{\tilde \nu}_{N_\ell}
	\Rightarrow
	\rmpm,
	\)
	which is equivalent to showing that for any fixed cylinder $[\bold{x}] = [x_1x_2\dots x_k]$
	\begin{equation}
	  {\tilde \nu}_{N_\ell}( [\bold{x}])
	  \to \rmpm( [\bold{x}]).
	  	  \label{eq:weaknur}
	\end{equation}

	Recall by \eqref{eq:ssimlimit}
\begin{equation*}
        \lim_{p \to \infty} \lambda^{-p}\ssim[x][y][p] = 
        {\hat \sigma}(y)
        {\hat \rho}(x).
      \end{equation*}
      Moreover, we have that 
      \[
\lambda^{-p}\ssim[x][y][p] = {\hat \sigma}(y){\hat \rho}(x) + O(e^{-cp}),
      \]
      uniformly in $x$ and $y$ by Perron-Frobenius theory and \eqref{eq:ssimlimit}. 

      By the convergence of $\Psi_a^r(N_\ell),$ we therefore have that for every $r \in \NN,$ there is an $\ell_0(r)$ sufficiently large so that for all $\ell > \ell_0$
      \[
	\sum_{q=1}^r \sum_{j=1}^{k_q(\ell)-1} 
                \frac{\ssim[*][(v_q(\ell),j)][p-q+1]}{\lambda^p}
		=
                \sum_{q=1}^r \sum_{1 \leq j < \kappa_q} 
		{\hat \sigma}( (\rho_q,j) )\lambda^{1-q}
		+ O(e^{-cp} + \lambda^{-r}),
      \]
      where we have used that $1 = \sum_{x \in \statespace}{ \hat\rho(x)}.$
      Furthermore, we have that
      \[
	\sum_{q=r+1}^{p(\ell)} \sum_{j=1}^{k_q(\ell)-1} 
                \frac{\ssim[*][(v_q(\ell),j)][p-q+1]}{\lambda^p}
		= O(\lambda^{-r}),
      \]
      uniformly in $N_\ell.$
      By \eqref{eq:>1nu}, we have that
      \[
	N_\ell = \sum_{q=1}^{p} \sum_{1 \leq j < k_q} {\ssim[*][(v_q,j)][p-q+1]},
      \]
      and hence 
      \[
	\frac{N_\ell}{\lambda^{p(\ell)}} = 
\sum_{q=1}^r \sum_{1 \leq j < \kappa_q} 
		{\hat \sigma}( (\rho_q,j) )\lambda^{1-q}
		+
		O(e^{-cp(\ell)} + \lambda^{-r}),
	\]
	uniformly in $r$ for all $\ell > \ell_0(r).$
	Define
	\begin{equation}
	  \label{eq:>1Nl}
	  R =
	  \lim_{\ell \to \infty} \frac{N_\ell}{\lambda^{p(\ell)}}
	  =\sum_{q=1}^\infty \sum_{1 \leq j < \kappa_q} 
		{\hat \sigma}( (\rho_q,j) )\lambda^{1-q}.
	 \end{equation}
	Then $\mathfrak{a}( (v,k) )$ is given by
	 \[
\mathfrak{a}( (v,k) ) = 
\frac{1}{R} {\hat \sigma}((v,k)) \sum_{q=1}^\infty \one[ v = \rho_q \text{ and } k < \kappa_q] \lambda^{1-q}, 
	 \]
	 and from \eqref{eq:>1Nl} and \eqref{eq:>1nu}, \eqref{eq:weaknur} follows.

Recall that
$U_f (\bold{x}) = \sum_{i=1}^{\infty} \lambda_f^{-i} \functional(\theta(v_i)_{< k_i}).$
This is a bounded $\C$-valued continuous function from $\pathspace[\infty]_r$ under the product topology.  Hence for any bounded uniformly continuous function $\phi : \C \to \R,$ 
by the definition of
\(
	{\tilde \nu}_{N_\ell}
	\Rightarrow
	\rmpm,
	\)
we have that
	 \[
	   \lim_{\ell \to \infty}
	   \int \phi(U_f(\bold{x})) {\tilde \nu}_{N_\ell}(d\bold{x})
	   =\int \phi(U_f(\bold{x})) {\rmpm}(d\bold{x})
	 \]

	 The remainder of the proof now proceeds in the same manner as the proof of Theorem~\ref{thm:<1}, with some minor changes. 
	 We define a function $\check{f}$ on $\statespace$ by
$\check{f}(v_i,k_i) = \functional(\theta(v_i)_{< k_i}).$
Then we have for any $1 \leq K \leq N_\ell,$ 
\[
  \lambda_f^{-p(\ell)} \functional (u_{<K}) = \sum_{i=1}^{p(\ell)} \lambda_f^{-i} \check{f}((v_i,k_i)),
\]
where $\Psi_a^r(K) = (v_1,k_1)(v_2,k_2)\cdots.$

For any natural number $p,$ let $U_{p} =\sum_{i=1}^{p} \lambda_f^{-i} \check{f}((v_i,k_i))$, which is a function on $\pathspace[\infty]_r.$  
Then we have that 
\begin{equation}
  \label{eq:Uconvergence}
\sup_{\bold{x} \in \pathspace[\infty]} |U_p(\bold{x}) - U_f(\bold{x})| \ll \lambda_f^{-p}.
\end{equation}

Then we have that for any bounded uniformly continuous $\phi,$
\begin{align*}
  \Exp(\phi(\lambda_f^{-p(\ell)} \functional (u_{<K_{N_\ell}})))
  &= \int \phi(U_{p(\ell)}(\bold{x})) {\tilde \nu}_{N_\ell}(d\bold{x}) \\
  &= \int \phi(U_{f}(\bold{x})) {\tilde \nu}_{N_\ell}(d\bold{x}) + o(1).
\end{align*}
Hence we have shown that 
\(
\lambda_f^{-p(\ell)} \functional (u_{<K_{N_\ell}})
\Rightarrow U_f(\bx),
\)
with $\bx$ distributed according to $\rmpm.$  
As 
\[
  \frac{N_\ell^{\log_\lambda|\lambda_f|}}{|\lambda_f|^{p(\ell)}} 
=\frac{\lambda^{p(\ell)\log_\lambda|\lambda_f|}}{|\lambda_f|^{p(\ell)}}
\frac{N_\ell^{\log_\lambda|\lambda_f|}}{\lambda^{p(\ell)\log_\lambda|\lambda_f|}}
\to R^{\log_\lambda|\lambda_f|},
\]
it follows that 
\[
\frac{\functional (u_{<K_{N_\ell}})}{N_\ell^{\log_\lambda|\lambda_f|}e^{ip(\ell)\arg\lambda_f}}
=
\frac
{|\lambda_f|^{p(\ell)}} 
{ N_\ell^{\log_\lambda|\lambda_f|}}
\frac{\functional (u_{<K_{N_\ell}})}
{\lambda_f^{p(\ell)}} 
\Rightarrow \frac{1}{R^{\log_\lambda|\lambda_f|}}U_f(\bx),
\]
and the proof is complete.

      \end{proof}

\section{Appendix: CLT}
In this section, we give a proof of the exact version of the Markov chain central limit theorem that we will need.  
Let $X_1,X_2,\dots$ be a primitive Markov chain on a finite state space $\statespace$ with invariant measure $\pi.$  Let $\Pr$ denote the probability measure of this Markov chain on $\statespace^{\N}.$
Let $p(x,y)$ be the transition matrix of the Markov chain, 
i.e.\,\(
p(x,y) = \Pr( X_2 = y\,|\,X_1 = x).
\)
Hence, $p(x,y)$ is a right stochastic matrix, and $\pi,$ its invariant measure, is the positive left Perron-Frobenius eigenvector of $p$ whose $\|\cdot\|_1$ norm is $1.$

We also define the reversed transition matrix $p^*$ given by $p^*(x,y) = \frac{p(y,x)\pi(y)}{\pi(x)}$ which is also a right stochastic matrix.  It is easily checked that if $Y_j = X_{n-j+1}$ for $j=1,2,\dots,n,$ then $\left( Y_j \right)_{j=1}^n$ are $n$ steps of a stationary Markov chain with transition matrix $p^*.$  Further, $p^*$ is the Hilbert space adjoint of $p$ with respect to the inner product on $\C^\statespace$ given by $\left( f, g \right)_\pi = \sum_{x \in \statespace} f(x)\overline{g(x)} \pi(x).$

We also define the operators $P$ and $P^*$ on $\C^\statespace$ by $(Ph)(x) = \sum_{x \in \statespace} p(x,y)h(y)$ and $(P^*h)(x) = \sum_{x \in \statespace} p^*(x,y)h(y).$
Let $\filt_N = \sigma(X_1,X_2,\ldots, X_N),$ the $\sigma$-algebra generated by the first $N$ states of the Markov chain.  Then we have that for any $N \in \NN,$
 $(Ph)(X_N) = \Exp(h(X_{N+1})|\filt_N)$.

\begin{Theorem}
        \label{thm:Markov}
Let $f : \statespace \to \C$ and $\lambda \in \mathbb{C}$ with $|\lambda|=1.$ Then either
\begin{enumerate}
  \item If there is a function $h : \statespace \to \C$ satisfying $P^*Ph = h$ and $f = \int f\,d\pi + h - \lambda Ph$ then we have
    \[
      \sup_{N \in \N} \left|\sum_{i=1}^N \lambda^i (f(X_i) -  \int f d \pi) \right| \leq 2\|h\|_{\infty} 
    \]
    almost surely.
        \item 
          Otherwise, if there is no such function, we have that
\[
        \frac{1}{\sqrt{N}} \left(\sum_{i=1}^N \lambda^i (f(X_i) -  \int f d \pi) \right) 
        \Rightarrow Z, 
\]
where $Z$ has a complex normal distribution with $\Exp|Z|^2 >0$.  If $\lambda \in \R$ and $f : \statespace \to \R,$ then $Z$ is a real normal distribution.  If $\lambda \not\in \R,$ then real and imaginary parts of $Z$ are independent and have identical variance.  Further, we always have that
    \[
	\Exp |Z|^2 = \Exp |g(X_1)|^2 + \sum_{k=2}^\infty 2\Exp\Re[\lambda^{k-1}g(X_1)\overline{g(X_k)}], 
	\]
	where $g(x) = f(x) - \int f\,d\pi.$
Finally, we have that there is a constant $C >0$ so that with probability $1$ 
\[
	\limsup_{N \to \infty} \frac{\left|\sum_{i=1}^N \lambda^i (f(X_i) -  \int f d \pi) \right|}
{\sqrt{N\log\log N}} = C.
\]

\end{enumerate}
\end{Theorem}

\begin{proof}
	We will use the martingale central limit theorem to prove the convergence (see \cite[Theorem 3.2]{HH}).  To do so, we will show that $Y_N = \sum_{i=1}^N \lambda^i g(X_i)$ is nearly a martingale.  The first step towards doing so is to show that there is a function $h(x)$ so that $g(x) = h(x) - \lambda (Ph)(x)$. Because $p$ is a primitive stochastic matrix, it has Perron-Frobenius eigenvalue $1$, and its other eigenvalues have modulus strictly less than $1$.
Hence, if $\lambda \ne 1$, then $I - \lambda P$ is invertible, so one can find a unique $h(x)$. If $\lambda=1$, then the space $W$ of $r(x)$ with $\int r(x) \pi (dx) =0$ is $\Im (I-P)$ since
\begin{enumerate}[(i)]
        \item $W$ has dimension $|\statespace| -1$, 
        \item $\Im(I-P)$ has dimension $|\statespace| -1$ by considering eigenvalues of $p$,
        \item $\Im(I-P) \subset W$: 
                \[\int h(x) \pi(dx) - \int (Ph)(x) \pi (dx) = \Exp h(X_1) - \Exp(\Exp(h(X_2)| \filt_1))=0.\]
\end{enumerate}
Note that the kernel of $I-P$ is just the constant functions, and hence we may choose $h(x)$ to have $\int h\,d\pi = 0.$  Having made this choice, $h$ is uniquely determined.  Notice that in the case that $\lambda \ne 1,$ we have that $h$ satisfies this condition as well.

Let $Z_N = \sum_{i=1}^{N-1} \lambda^{i+1}[h(X_{i+1}) - (Ph)(X_i)]$. 
Note that 
\[
        Y_N = Z_N + \lambda h(X_1) - \lambda^{N+1}(Ph)(X_N).
\]
Also we can see that $Z_N$ is a martingale:
\[
	\Exp(Z_{N+1}|\filt_N) = Z_N + \lambda^{N+1} \Exp(h(X_{N+1}) -(Ph)(X_N)|\filt_N) = Z_N.
\]
Furthermore, 
\[
        \Exp|Z_N|^2 = \sum_{i=1}^{N-1} \Exp|h(X_{i+1}) - (Ph)(X_i)|^2 
        = (N-1)\Exp|h(X_2) - (Ph)(X_1)|^2
\]
by the orthogonality of martingale increments.  We now show that this variance is $0$ if and only if the function $h$ satisfies $P^*Ph = h.$  If the variance is $0,$ then we have that $Z_N$ almost surely vanishes, and so the first conclusion of the theorem follows from the definition of $Y_N.$  Conversely if $\Exp|h(X_2) - (Ph)(X_1)|^2 \neq 0,$ the second conclusion of the theorem follows immediately from the martingale central limit theorem and the law of iterated logarithm for martingales \cite{Stout}. 

We now expand the squares to get
\begin{align}
\Exp|h(X_2) - (Ph)(X_1)|^2
&=\Exp(h(X_2) - (Ph)(X_1))(\overline{h}(X_2) - \overline{(Ph)}(X_1)) \nonumber \\
&=\Exp |h(X_1)|^2 - \Exp |(Ph)(X_1)|^2 \nonumber\\
&= 
\left( h,h \right)_{\pi}
-\left( Ph,Ph \right)_{\pi} \label{eq:PhPh}\\
&= 
\left( (I - P^*P)h,h \right)_{\pi}.\nonumber
      \end{align}
In the second equality, we have used that
\begin{align*}
        \Exp (\overline{h}(X_2) (Ph)(X_1)) 
        &= \Exp (\Exp (\overline{h}(X_2)| \filt_1)(Ph)(X_1)) \\
        &= \Exp ((\overline{Ph})(X_1) (Ph)(X_1)) \\
        &= \Exp |(Ph)(X_1)|^2. 
\end{align*}
The operator 
\(
I - P^*P
\)
is Hermitian positive semidefinite, and hence the variance is $0$ if and only if 
\(
(I - P^*P)h = 0.
\)
In the case that $\lambda \neq 1,$ the solution to $(I-\lambda P)h = g$ is unique, and hence we are done as this the only possible $h$ that could satisfy the criterion in (1).  If $\lambda = 1,$ the collection of $h'$ for which $(I-\lambda P)h' = g$ just differ from $h$ by constant functions.  Hence, $P^*Ph'=h'$ if and only if $P^*Ph=h.$
It remains to show the formula for the limiting variance of $Z.$  We have that 
\[
	\Exp |Z|^2 = \lim_{N\to\infty} \frac{\Exp |Z_N|^2}{N}
	= \Exp |h(X_2) - (Ph)(X_1)|^2.
\]
By \eqref{eq:PhPh}, we therefore have that 
\begin{align*}
\Exp |Z|^2
&=
\left( h,h \right)_{\pi}
-\left( Ph,Ph \right)_{\pi} \\
&=
\left( h-\lambda Ph,h \right)_{\pi}
+\left( \lambda Ph,h-\lambda Ph \right)_{\pi} \\
&=
\left( g ,h \right)_{\pi}
+\left( \lambda Ph,g \right)_{\pi} \\
&=
\left( ((I-\overline{\lambda}P^{*})^{-1} + \lambda P(I - \lambda P)^{-1})  g ,g \right)_{\pi}. \\
\intertext{These inverses always exist on the space of functions $W$.  Expanding the inverses as Neumann series,
we arrive at}
\Exp |Z|^2
&=(g,g)_{\pi} + \sum_{k=1}^\infty \left[ 
	(\overline{\lambda}^k P^{*k}g,g)_{\pi}
	+ ({\lambda}^k P^{k}g,g)_{\pi}
\right].
\end{align*}
The desired formula for the variance now follows using the identities $P^k g(x) = \Exp\left[ g(X_{k+1}) \vert X_1 = x \right]$ and $P^{*k}g(x) = \Exp\left[ g(X_{1}) \vert X_{k+1}= x \right].$

To show that the real and imaginary parts of $Z$ are independent and have the same variance in the case $\lambda \not\in \R,$ observe that, again by the orthogonality of martingale increments,
\[
	\Exp Z_N^2 = \sum_{i=1}^{N-1} \Exp\lambda^{2(i+1)}(h(X_{i+1}) - (Ph)(X_i))^2 
	= \Exp(h(X_{2}) - (Ph)(X_1))^2
	\sum_{i=1}^{N-1} \lambda^{2(i+1)}.
\]
In particular, if $\lambda \not\in \R,$ we have that 
\[
	\Exp Z^2 = \lim_{N\to\infty} \frac{\Exp Z_N^2}{N}
	=0.
\]
This means that
\[
	0 = \Exp Z^2 = 
	\Exp(\Re Z)^2
	-\Exp(\Im Z)^2
	+2i \Exp(\Re Z \Im z).
\]
As $(\Re Z, \Im Z)$ are jointly Gaussian and their covariance is $0,$ they are independent.  Further, the variances of the real and imaginary parts match.

\end{proof}

\section*{Acknowledgments}
The authors would like to thank Omri Sarig for useful discussion and comments.

\bibliographystyle{plain}
\bibliography{Substitution}

\begin{thebibliography}{10}

\bibitem{A}
Boris Adamczewski.
\newblock Symbolic discrepancy and self-similar dynamics.
\newblock {\em Ann. Inst. Fourier (Grenoble)}, 54(7):2201--2234 (2005), 2004.

\bibitem{ADDS}
A.~Avila, D.~Dolgopyat, E.~Duryev, and O.~Sarig.
\newblock The visits to zero of a random walk driven by an irrational rotation.
\newblock {\em Israel Journal of Mathematics}, pages 1--65, 2015.

\bibitem{Beck1}
J{\'o}zsef Beck.
\newblock Randomness of the square root of 2 and the giant leap, {P}art 1.
\newblock {\em Period. Math. Hungar.}, 60(2):137--242, 2010.

\bibitem{Beck2}
J{\'o}zsef Beck.
\newblock Randomness of the square root of 2 and the giant leap, {P}art 2.
\newblock {\em Period. Math. Hungar.}, 62(2):127--246, 2011.

\bibitem{BH}
Mike Boyle and David Handelman.
\newblock The spectra of nonnegative matrices via symbolic dynamics.
\newblock {\em Ann. of Math. (2)}, 133(2):249--316, 1991.

\bibitem{BBH}
Xavier Bressaud, Alexander~I. Bufetov, and Pascal Hubert.
\newblock Deviation of ergodic averages for substitution dynamical systems with
  eigenvalues of modulus 1.
\newblock {\em Proc. Lond. Math. Soc. (3)}, 109(2):483--522, 2014.

\bibitem{Dudley}
R.~M. Dudley.
\newblock Convergence of {B}aire measures.
\newblock {\em Studia Math.}, 27:251--268, 1966.

\bibitem{Fogg}
N.~Pytheas Fogg.
\newblock {\em Substitutions in dynamics, arithmetics and combinatorics},
  volume 1794 of {\em Lecture Notes in Mathematics}.
\newblock Springer-Verlag, Berlin, 2002.
\newblock Edited by V. Berth{\'e}, S. Ferenczi, C. Mauduit and A. Siegel.

\bibitem{GH}
Walter~Helbig Gottschalk and Gustav~Arnold Hedlund.
\newblock {\em Topological dynamics}.
\newblock American Mathematical Society Colloquium Publications, Vol. 36.
  American Mathematical Society, Providence, R. I., 1955.

\bibitem{HH}
P.~Hall and C.~C. Heyde.
\newblock {\em Martingale limit theory and its application}.
\newblock Academic Press, Inc. [Harcourt Brace Jovanovich, Publishers], New
  York-London, 1980.
\newblock Probability and Mathematical Statistics.

\bibitem{Host}
B.~Host.
\newblock Valeurs propres des syst\`emes dynamiques d\'efinis par des
  substitutions de longueur variable.
\newblock {\em Ergodic Theory Dynam. Systems}, 6(4):529--540, 1986.

\bibitem{KOR}
Ki~Hang Kim, Nicholas~S. Ormes, and Fred~W. Roush.
\newblock The spectra of nonnegative integer matrices via formal power series.
\newblock {\em J. Amer. Math. Soc.}, 13(4):773--806 (electronic), 2000.

\bibitem{Klenke}
Achim Klenke.
\newblock {\em Probability theory}.
\newblock Universitext. Springer, London, second edition, 2014.
\newblock A comprehensive course.

\bibitem{KL}
Isaac Kornfeld and Michael Lin.
\newblock Coboundaries of irreducible {M}arkov operators on {$C(K)$}.
\newblock {\em Israel J. Math.}, 97:189--202, 1997.

\bibitem{LPW}
David~A. Levin, Yuval Peres, and Elizabeth~L. Wilmer.
\newblock {\em Markov chains and mixing times}.
\newblock American Mathematical Society, Providence, RI, 2009.
\newblock With a chapter by James G. Propp and David B. Wilson.

\bibitem{Livshits}
A.~N. Livshits.
\newblock Sufficient conditions for weak mixing of substitutions and of
  stationary adic transformations.
\newblock {\em Mat. Zametki}, 44(6):785--793, 862, 1988.

\bibitem{Parry}
William Parry.
\newblock Intrinsic {M}arkov chains.
\newblock {\em Trans. Amer. Math. Soc.}, 112:55--66, 1964.

\bibitem{Q}
Martine Queff{\'e}lec.
\newblock {\em Substitution dynamical systems---spectral analysis}, volume 1294
  of {\em Lecture Notes in Mathematics}.
\newblock Springer-Verlag, Berlin, second edition, 2010.

\bibitem{Stout}
William~F. Stout.
\newblock A martingale analogue of {K}olmogorov's law of the iterated
  logarithm.
\newblock {\em Z. Wahrscheinlichkeitstheorie und Verw. Gebiete}, 15:279--290,
  1970.

\end{thebibliography}

\end{document}